\documentclass[letter,11pt]{amsart}

\usepackage{color}

\usepackage{amsfonts,amsmath, amssymb,amsthm,amscd, enumerate}
\usepackage{bm}
\usepackage{mathrsfs}
\usepackage{yfonts}
\usepackage{verbatim}
\usepackage{hyperref}
\usepackage{graphicx}
\usepackage[latin1]{inputenc}
\usepackage{latexsym}
\usepackage{lscape}
\usepackage{subfigure}
\usepackage{setspace}

\usepackage[margin=1in]{geometry}

\newcommand{\MM}{\ensuremath{\mathbb{MM}^{\mathbf a}}}

\newcommand{\MP}{\ensuremath{\mathbb{MP}^{\mathbf a}}}

\newcommand{\WP}{\ensuremath{\mathbb{QWP}^{\mathbf a}}}

\newcommand{\Noumi}[2]{\ensuremath{\mathfrak{N}^{#1}_{#2}}}

\newcommand{\N}{(\mathbb{Z}_{\ge 0})}

\newtheorem{theorem}{Theorem}[section]

\newtheorem{lemma}[theorem]{Lemma}
\newtheorem{proposition}[theorem]{Proposition}
\newtheorem{corollary}[theorem]{Corollary}

\theoremstyle{definition}

\theoremstyle{definition}

\theoremstyle{definition}
\newtheorem{definition}[theorem]{Definition}

\theoremstyle{definition}

\newcommand{\HH}{{\mathrm H}}
\newcommand{\W}{{\mathrm W}}

\newcommand{\J}{J_{N,\{a_i\},\rho}}
\newcommand{\K}{K_{u,N,\{a_i\},\rho}}
\newcommand{\Ktilde}{ \tilde{K}_{u,N,\{a_i\},\rho}}
\newcommand{\Kprime}{K'_{u,N,\{a_i\},\rho}}
\newcommand{\G}{{\mathrm G_{N,\{a_i\},\rho}}}

\renewcommand{\O}{\mathcal E}

\newcommand{\blambda}{{\boldsymbol\lambda}}
\newcommand{\cR}{\ensuremath{\mathcal{R}}}

\renewcommand{\l}{\Big \langle}
\renewcommand{\r}{\Big \rangle}

\newcommand{\Y}{\mathbb Y}

\def\Gampq{\Gamma_{\!p,q}}
\def\Gamq{\tilde{\Gamma}_{\!q}}

\newcommand{\MF}{\mathbb{MM}^{\mathbf f}}
\newcommand{\MPF}{\mathbb{MP}^{\mathbf f}}

\newcommand{\ldeg}{{\rm ldeg}}
\newcommand{\Z}{\ensuremath{\mathbb{Z}}}

\newcommand{\Mac}{\mathfrak{M}}

\renewcommand{\i}{\mathbf i}

\numberwithin{equation}{section} \numberwithin{theorem}{section}

\title{Observables of Macdonald processes}

\author[A. Borodin]{Alexei Borodin}
\address{A. Borodin,
Massachusetts Institute of Technology, Department of Mathematics, Cambridge, MA 02139, USA, and
Institute for Information Transmission Problems of Russian Academy of Sciences, Moscow, Russia}
\email{borodin@math.mit.edu}

\author[I. Corwin]{Ivan Corwin}
\address{I. Corwin,
Massachusetts Institute of Technology, Department of Mathematics, Cambridge, MA 02139, USA, and Clay Mathematics Institute, 10 Memorial Blvd. Suite 902, Providence, RI 02903, USA}
\email{icorwin@mit.edu}

\author[V. Gorin]{Vadim Gorin}
\address{V. Gorin,
Massachusetts Institute of Technology, Department of Mathematics, Cambridge, MA 02139, USA, and
Institute for Information Transmission Problems of Russian Academy of Sciences, Moscow, Russia}
\email{vadicgor@gmail.com}

\author[S. Shakirov]{Shamil Shakirov}
\address{S. Shakirov,
U.C. Berkeley, Department of Mathematics, Berkeley, CA 94720 USA} \email{shakirov@itep.ru}

\begin{document}

\maketitle

\begin{abstract}
We present a framework for computing averages of various observables of Macdonald processes. This
leads to new contour--integral formulas for averages of a large class of multilevel observables, as
well as Fredholm determinants for averages of two different single level observables.
\end{abstract}

\begin{flushright}
To the memory of A.Zelevinsky.
\end{flushright}

\tableofcontents

\section{Introduction}
The last decade saw great success surrounding the applications of Schur processes \cite{Ok},
\cite{OR_Schur} to probability (cf.\ \cite{BG}). Starting with the 2011 work of \cite{BigMac} (see
also \cite{F1}, \cite{FR}), more general Macdonald processes  have proved
useful in solving a number of problems
in probability, including: computing exact Fredholm determinant formulas and associated asymptotics
for one--point marginal distributions of the O'Connell--Yor semi--discrete directed polymer
\cite{BigMac}, \cite{BCF} (see also \cite{OConYor}, \cite{OCon}), log--gamma discrete directed
polymer \cite{BigMac}, \cite{BCR} (see also \cite{COSZ}, \cite{S}), Kardar--Parisi--Zhang /
stochastic heat equation \cite{BCF} (see also \cite{ACQ}, \cite{SS}), $q$--TASEP \cite{BigMac},
\cite{BCS}, \cite{BC_discrete} and $q$--PushASEP \cite{BP}, \cite{CP}; showing Gaussian free field
fluctuations for the general $\beta$ Jacobi corners process \cite{BG_GFF} and constructing a
multilevel extension of the general $\beta$ Dyson Brownian Motion \cite{GS}.

These probabilistic systems and formulas describing them arise under various choices and limits of
parameters (sometimes called {\it degenerations}) for Macdonald processes (as well as natural
dynamics which behave well with respect to Macdonald processes). There are other important
degenerations including the study of measures on plane partitions \cite{Vul}, random unitriangular
matrices over finite fields \cite{Borodin}, \cite{Fulman}, \cite[Section 4]{GKV}, Kingman and
Ewens--Pitman partition structures \cite{Kingman},  \cite{Kerov}, \cite[Chapter I]{Kerov_book},
\cite{Petrov}, $z$--measures as well as other distributions originating from the representation
theory of ``big'' groups \cite{BO-z}, \cite{OkOl}, \cite{KOO}. Many more examples are known for
the degeneration related to the Schur processes, e.g.\ domino/lozenge tilings and shufflings,
totally asymmetric simple exclusion process,  polynuclear growth model, last passage percolation,
longest increasing subsequences in random permutations (see the review \cite{BG}). Figure
\ref{BigPicture}  indicates how these systems relate to Macdonald processes.

\medskip

\begin{figure}
\begin{center}
\includegraphics[scale=.68]{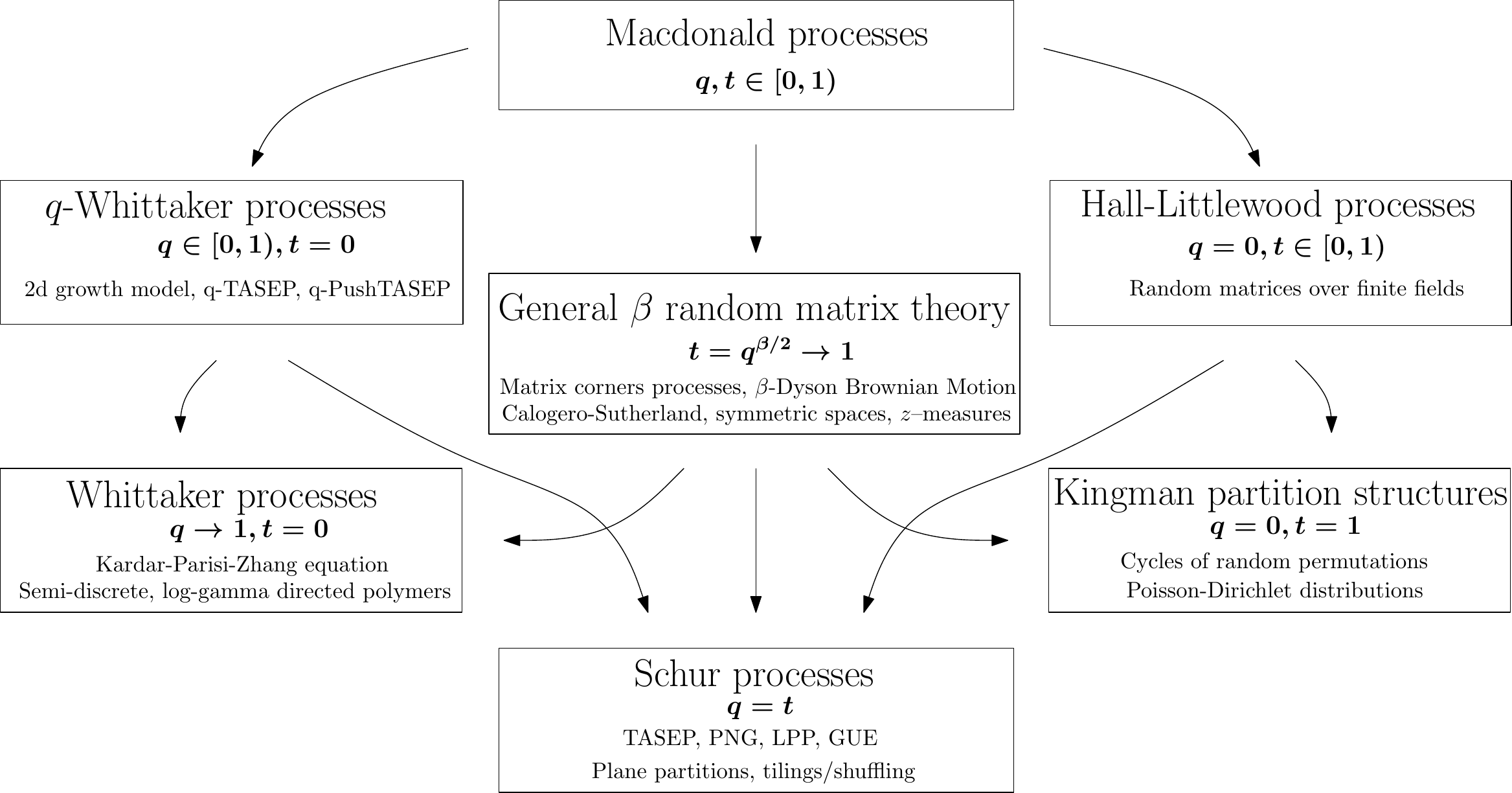}
\end{center}
\caption{Macdonald processes unify the study of a diverse array of probabilistic systems.}\label{BigPicture}
\end{figure}

The integrable properties of the Macdonald symmetric functions (i.e., the family of Macdonald
difference operators diagonalized by them) through which the Macdonald processes are defined
naturally, lead to a family of observables whose expectations can be concisely written via
contour--integral formulas. This approach to studying observables of Macdonald processes was
initiated in \cite{BigMac} and the purpose of this paper is to develop this direction in its full
generality. The family of observables for whose expectation we have contour--integral formulas is
sufficiently rich so as to completely characterize the distribution of the Macdonald process.
Thus, one could call this an {\it integrable probabilistic system} (cf.\ \cite{BG}). We expect
that these new results will prove useful for many of the degenerations of Macdonald processes
indicated in Figure \ref{BigPicture}.

This should be compared to the fact that Schur processes of \cite{OR_Schur} (degenerations of
Macdonald processes when $q=t$, see Figure \ref{BigPicture}) are known to be determinantal, meaning
that all of their correlation functions are given by determinants made of a single correlation
kernel. Marginal distributions of determinantal point processes are known to be expressible in
terms of Fredholm determinants. Macdonald processes do not appear to be determinantal and the
family of (non-local) observables which we study at the Macdonald processes level is different from
those related to correlation functions (and do not degenerate to those when $q=t$)\footnote{It is
possible at $q=t$ to use the Macdonald process observables to recover the Schur process correlation
kernel, cf. \cite[Remark 2.2.15]{BigMac} and \cite{Amol}.}. Nevertheless we introduce two single
level observables of Macdonald processes whose expectations are given by Fredholm determinants. The
first relies upon an operator of \cite{Noumi} (see also \cite{FHHSY}), which is diagonalized by the
Macdonald symmetric polynomials, while the second (more elementary result) relies upon the
Macdonald difference operators.

Besides providing a unified framework through which to study (and discover) a variety of probabilistic systems,
Macdonald processes exist at a sufficiently high algebraic level so that they may be treated as formal algebraic
objects. This formal perspective, which is introduced in Section \ref{gendef} and applied in detail in Section
\ref{Section_Formal_Macdonald}, enables us to deal with a more general case of Macdonald processes than just the
ascending Macdonald processes which was the primary interest of \cite{BigMac}.  Already in \cite{BCFV} this formal
perspective has proved useful in justifying identities for which a direct analytic proof is unjustifiable due to
divergences.  In Section \ref{ascsec} we do specialize to ascending Macdonald processes and find new formulas for
expectations of multilevel observables, some of which have already been applied in work on $q$--TASEP \cite{BCS} and
general $\beta$ random matrix theory \cite{BG}.

\medskip

We briefly introduce Macdonald processes and then highlight two of the results which we prove in
subsequent sections. The notation and exact definitions related to symmetric functions and formal
power series is introduced and explained in Section \ref{gendef}. As given in Definition
\ref{Definition_mac_pro}, the {\it formal Macdonald process} $\MPF_{N,\mathbf{A},\mathbf{B}}$ is a
formal probability measure (of total weight 1) on $\Y^N$ (here $\Y$ is the set of all partitions)
such that
$$
\MPF_{N,\mathbf{A},\mathbf{B}}(\lambda^1,\dots,\lambda^N) = \frac{P_{\lambda^1}(A^1) \Psi_{\lambda^2,\lambda^{1}}(A^2;B^{1})
\Psi_{\lambda^3,\lambda^{2}}(A^3;B^{2}) \cdots  \Psi_{\lambda^N,\lambda^{N-1}}(A^N;B^{N-1})Q_{\lambda^N}(B^N)}{ \prod_{1\le \alpha\le \beta \le N} \Pi(A^\alpha; B^\beta)}
,
$$
where
$$
\Psi_{\lambda,\mu}(A;B)=\sum_{\nu\in\Y} P_{\lambda/\nu}(A) Q_{\mu/\nu}(B),
$$
and the factors $\Pi(A^\alpha; B^\beta)$ are defined via
$$
\Pi(X;Y)= \sum_{\lambda\in \Y} P_{\lambda}(X)Q_{\lambda}(Y).
$$
Here $P_{\bullet}$ and $Q_{\bullet}$
 are (skew)
Macdonald symmetric functions (cf. Definition \ref{Macdef}) and they depend on
two auxiliary (Macdonald) parameters traditionally denoted by $q,t~\in~[0,1)$.

This is called a formal probability measure since it does not assign a non--negative real
probability to a given choice of $\lambda^1,\ldots, \lambda^N$, but rather assigns a formal power
series in the symmetric functions of the $2N$ sets of variables $\mathbf{A}=(A^1,\dots,A^N)$,
$\mathbf B=(B^1,\dots,B^N)$ (each of these sets of variables is, itself, an infinite collection of
indeterminates, so that e.g.\ $A^1 = (a^1_1,a^1_2,\ldots)$). Alternatively, this can be thought of
as formal power series in the Newton power sums $p_k(A^{i})$ and $p_{k}(B^{j})$, where $p_k(X) =
\sum_{i} (x_i)^k$.

Define the observable $\O_r: \Y\to \mathbb C$ as (cf. Definition \ref{Definition_observable})
$$
\O_r (\lambda)= \lim_{N\to\infty} e_r(q^{-\lambda_1},q^{-\lambda_2}t,\dots, q^{-\lambda_N}
t^{N-1}),\quad r\ge 1,
$$
where $e_r$ is the $r^{th}$ elementary symmetric polynomial, and $\O_0 (\lambda)=1$.

For sets of indeterminates $X=(x_1,x_2,\ldots)$ and $Y=(y_1,y_2,\ldots)$, define also
$$
 \HH(X;Y) =\prod_{i,j} \frac{1-t x_i y_j}{1-x_i y_j}, \qquad \textrm{and} \qquad
 \W(X;Y) =\prod_{i,j} \frac{(1-t x_i y_j) (1-q x_i y_j)}{(1-x_i y_j)(1-qt x_i y_j)}.
$$
These can be viewed as formal power series via $(1-u)^{-1} = \sum_{j\geq 0} u^j$.

The statement of Theorem \ref{1.1} below (Theorem \ref{theorem_observable_multilevel} in the main text) should be
understood formally as an identity of symmetric power series.

\begin{theorem}\label{1.1}
 Take $N\ge 1$ and $r_1,\dots,r_N\ge 0$. For $1\le m \le N$, set $V^m=\{v^m_1,\dots,v^m_{r_m}\}$  and define
$$
 DV^m = \frac{1}{(r_m)!(2\pi \i)^{r_m}} \det \left[ \frac{1}{
 v^m_i-tv^m_j}\right]_{i,j=1}^{r_m} \prod_{i=1}^{r_m} dv_i^m.
$$
We have
\begin{multline*}
\label{eq_x6}
 \sum_{\lambda^1,\dots,\lambda^N\in\Y}
 \O_{r_1}(\lambda^1) \cdots
 \O_{r_N}(\lambda^N) \MPF_{N,\mathbf{A},\mathbf{B}} (\lambda^1,\dots,\lambda^N)
 \\= \oint \cdots \oint \prod_{m=1}^N
 (DV^m) \prod_{1\le \alpha\le \beta\le N} \HH\big((qV^\alpha)^{-1}; B^\beta\big)
  \HH\big(A^{\alpha};V^\beta\big)
  \W\big((qV^\alpha)^{-1}; V^\beta\big).
\end{multline*}
Here for a set of variables $V=\{v_1,\dots,v_r\}$,  $(qV)^{-1}$ means the set $\{(qv_1)^{-1},\dots,(q v_r)^{-1}\}$. The
contours of integration are  a collection of positively oriented circles $\gamma_1,\dots,\gamma_m$ of
radii $R_1,\dots,R_m$ around the origin such that $v_i^\alpha$ is integrated over $\gamma_\alpha$,
and the radii are such that $R_\beta<qR_\alpha$ for $1\le \alpha<\beta \le N$.
\end{theorem}

In what follows we call a homomorphism of an algebra into $\mathbb C$ \emph{a specialization}. One example of a
specialization of the algebra of symmetric functions in variables $x_1,x_2,\dots$, is obtained by substituting complex
numbers (subject to certain convergence conditions) in place of $x_i$, $i=1,2,\dots$.

Applying appropriate speacializations to the identity of Theorem \ref{1.1} we get an analytic statement.
However, the {\it only} proof we know of the analytic identity proceeds through the formal setting (when restricted to
ascending Macdonald processes as discussed below, a direct analytic proof is known).

\medskip

The ascending Macdonald process $\MP_{N;\{a_i\};\rho}$ (cf. Definition
\ref{def_Mac_ascending_proc}) is the result of specializing the $2N$ sets of variables
$\mathbf{A}=(A^1,\dots,A^N)$, $\mathbf{B}=(B^1,\dots,B^N)$ in a certain way. This is now a
(possibly complex-valued) measure on sequences of interlacing partitions
$\lambda^1,\ldots,\lambda^{N}$ so that $\lambda^{i}$, $1\leq i\leq N$, has at most $i$ nonzero
parts, and
 $$
  \lambda^i_1\ge \lambda^{i-1}_2 \ge \lambda^i_2 \ge \dots\ge \lambda^{i-1}_{i-1} \ge
  \lambda^{i}_i
 $$
for $2\leq i\leq N$. If the specializations are assumed to have certain positivity properties, then the ascending
Macdonald process becomes a bona--fide probability measure, cf.\ \cite[Section 2.2]{BigMac}. In Section
\ref{Section_ascending_obs} we provide contour--integral formulas for expectations of observables of the ascending
Macdonald process of the form (in Theorem \ref{theorem_ascending_multilevel_1} and  Theorem
\ref{theorem_observable_with_vars_inversed}, respectively)
$$
\prod_{i=1}^m e_{r_i}(q^{\lambda^{n_i}_1}t^{n_i-1},q^{\lambda^{n_i}_2}t^{n_i-2},\dots,q^{\lambda^{n_i}_{n_i}}), \qquad \textrm{and} \qquad
\prod_{i=1}^m e_{r_i}(q^{-\lambda^{n_i}_1}t^{1-n_i},\dots,q^{-\lambda^{n_i}_{n_i}}),
$$
where $N\geq n_1\geq n_2\geq \cdots \geq n_m\geq 1$ and $r_1,\ldots ,r_m$ are such that $0\leq r_i\leq n_i$ for $i=1,\ldots, m$.

The projection of the ascending Macdonald process to $\lambda^N$ is the Macdonald measure $\MM_{N;\{a_i\};\rho}$, which is a complex-valued measure on $\Y$ which sums (over $\lambda\in \Y$) to $1$ such that (replacing $\lambda^N$ by $\lambda$)
 $$
  \MM_{N;\{a_i\};\rho}(\lambda)=\dfrac{
  P_{\lambda}(a_1,\dots,a_N) Q_{\lambda}(\rho)}{\Pi(a_1,\dots,a_N;\rho)}.
 $$
Here $a_1,\ldots, a_N$ are complex numbers and $Q_{\lambda}(\rho)$ is the specialization of
$Q_{\lambda}$ (cf. Section \ref{section_formal_measures}). Given some assumptions on the $\{a_i\}$
and $\rho$, the normalizing term
$$
\Pi(a_1,\dots,a_N;\rho):= \sum_{\lambda\in \Y} P_{\lambda}(a_1,\ldots,a_N) Q_{\lambda}(\rho)
$$
is finite and the measure is well--defined.

The $t=0$ degeneration of Theorem \ref{1.2} below (Theorem \ref{Theorem_qtdet} in the main text)
was previously discovered in \cite[Corollary 3.2.10 and Theorem 3.2.11]{BigMac} and served as the
basis for computing exact Fredholm determinant formulas and associated asymptotics for one--point
marginal distributions of the O'Connell--Yor semi--discrete directed polymer,
Kardar--Parisi--Zhang / stochastic heat equation and $q$--TASEP. The proof in \cite{BigMac} relied
on the first Macdonald difference operator and its powers. Our present result uses a different
operator diagonalized by the Macdonald polynomials.

\begin{theorem}\label{1.2}
Fix $N$ non--zero complex numbers $a_1$,\dots, $a_N$ and a specialization $\rho$. Then, under certain assumptions (cf.\
 Theorem \ref{Theorem_qtdet}) on these parameters, as well as the contour $\gamma$, we have that the following
equality holds as an identity of power series in $u$:
\begin{equation*}
\sum_{\lambda\in\Y} \prod_{i=1}^{N}
\frac{\big(q^{\lambda_i}t^{N-i+1}u;q\big)_{\infty}}{\big(q^{\lambda_i}t^{N-i}u;q\big)_{\infty}}\,
\MM_{N,\{a_i\},\rho}(\lambda) = \det\big(I+\K\big)_{L^2(\gamma)},
\end{equation*}
where
\begin{equation*}
\K(w,w') =  \sum_{v=1}^{\infty} \frac{u^v}{q^v w-w'} \frac{\G(w)}{\G(q^v w)},\qquad
\G(w) = \frac{1}{\Pi(w;\rho)} \prod_{j=1}^{N} \frac{(tw/a_j;q)_{\infty}}{(w/a_j;q)_{\infty}}.
\end{equation*}
\end{theorem}

For the degeneration related to directed polymers (i.e.\ Whittaker processes, see Figure
\ref{BigPicture}), this (somewhat surprisingly) converges as $N\to \infty$ to the GUE Tracy-Widom
distribution \cite{BigMac}, \cite{BCF}, \cite{BCR}. We look forward to investigating the
asymptotics of the above Fredholm determinant in other degenerations indicated in Figure
\ref{BigPicture}.

\subsection{Acknowledgements}
AB was partially supported by the NSF grant DMS-1056390. IC was partially supported by the NSF
through DMS-1208998 as well as by Microsoft Research through the Schramm Memorial Fellowship, and
by a Clay Mathematics Institute Research Fellowship. VG was partially supported by RFBR--CNRS
grant 11-01-93105. SS was supported in part by grant RFBR 13-02-00478, grant for support of scientific schools NSH-3349.2012.2, and government contract 8606.

\section{General definitions}\label{gendef}

\subsection{Symmetric functions}

A \emph{partition} $\lambda$ is a weakly decreasing sequence of non--negative integers
$\lambda_1\ge\lambda_2\ge\dots\ge 0$, such that $\sum_i \lambda_i<\infty$. The last sum is called
the size of a partition and is denoted $|\lambda|$. Let $\Y_n$ denote the set of all partitions of
size $n$ and set
$$
 \Y=\bigcup_{n=0}^\infty \Y_n,
$$
where we assume $\Y_0$ to be a singleton consisting of $\varnothing$. The number of nonzero
coordinates (parts) in $\lambda$ is called the length of $\lambda$ and denoted $\ell(\lambda)$.

In what follows we denote by capital letters $X,Y,A,B$, \emph{sets} of variables and by lower
case letters $x,y,\dots$, \emph{single} variables. Let $\Lambda_X$ denote the $\mathbb Z_{\ge 0}$
graded algebra (over $\mathbb C$) of symmetric functions in variables $X=(x_1,x_2,\dots)$, which
can be viewed as the algebra of symmetric polynomials in infinitely many variables $x_1,x_2,\dots$
of bounded degree, see e.g.\ \cite[Chapter 1]{M} for general information on $\Lambda$. One way to
view $\Lambda$ is as an algebra of polynomials in Newton power sums
$$
 p_k(X)=\sum_{i} (x_i)^k, \quad k=1,2,\dots.
$$
For any partition $\lambda$ we set
$$
 p_\lambda(X)=\prod_{i=1}^{\ell(\lambda)} p_{\lambda_i}(X).
$$
Elements $p_\lambda(X)$, $\lambda\in\Y$ form a linear basis in $\Lambda_X$.

An alternative set of algebraically independent generators of $\Lambda_X$ is given by the
elementary symmetric functions
$$
 e_k(X)=\sum_{i_1<i_2<\dots<i_k} x_{i_1} x_{i_2}\cdots x_{i_k},\quad k=1,2,\dots.
$$
We usually write $\Lambda_X$, $\Lambda_Y$, etc.\ for (isomorphic) algebras of symmetric functions
in variables $X=(x_1,x_2,\dots)$, $Y=(y_1,y_2,\dots)$ and so on. When the set of variables is
irrelevant, we omit it from the notations and write simply $\Lambda$.

For a symmetric function $f$ let $\phi_0(f)$ be its free (constant, degree $0$) term. Clearly
$\phi:\Lambda\to \mathbb C$ is an algebra homomorphism, and $\phi_0(p_k)=0$, $k=1,2,\dots$.

In what follows we fix two parameters $q$, $t$ and assume  that they  are real numbers satisfying
$0<q,t<1$. Alternatively, in many places below we could have instead assumed that $q$ and $t$ are
formal variables, replacing $\mathbb C$ in the definition of $\Lambda$ and all the following
definitions with the algebra $\mathbb C(q,t)$ of rational function in $q$ and $t$. Since $q$ and
$t$ never change throughout the paper, we omit the dependence on them from the notations.

The Macdonald scalar product $\l \cdot,\cdot\r$ on $\Lambda$ is defined via
\begin{equation}
\label{eq_Macd_scalar} \l p_\lambda, p_\mu \r = \delta_{\lambda,\mu}
\left(\prod_{i=1}^{\ell(\lambda)}\frac{1-q^{\lambda_i}}{1-t^{\lambda_i}}
\right)\left(\prod_{i=1}^\infty i^{m_i(\lambda)} m_i(\lambda)!\right),
\end{equation}
where $m_i(\lambda)$ is the number of parts in $\lambda$ equal to $i$.

The following definition can be found in \cite[Chapter VI]{M}.

\begin{definition}\label{Macdef} Macdonald symmetric functions $P_\lambda$, $\lambda\in \Y$ are a
unique linear basis in $\Lambda$ such that
\begin{enumerate}
\item $\l P_\lambda, P_\mu \r=0$ unless $\lambda=\mu$.
\item The leading
(with respect to the reverse lexicographic order, i.e., $x_1^n$ is the largest monomial of degree
$n$) monomial in $P_\lambda$ is $
\prod_{i=1}^{\ell(\lambda)} x_i^{\lambda_i}.
$
\end{enumerate}
\end{definition}
{\bf Remark 1.} The Macdonald symmetric function $P_\lambda$ is a homogeneous symmetric function of
degree $|\lambda|$.

 {\bf Remark 2.} If we set $x_{N+1}=x_{N+2}=\dots=0$ in $P_\lambda(X)$ then we
arrive at symmetric polynomials $P_\lambda(x_1,\dots,x_N)$ in $N$ variables, which are called the
Macdonald polynomials.

\medskip

Macdonald symmetric functions $Q_\lambda$, $\lambda\in \Y$ are dual to $P_\lambda$, with respect to
the Macdonald scalar product:
$$
 Q_\lambda=\l P_\lambda, P_\lambda \r^{-1} P_\lambda,\quad \l P_\lambda, Q_\lambda \r
 =\delta_{\lambda,\mu}, \quad \lambda,\mu\in\Y.
$$

We also need \emph{skew} Macdonald symmetric functions (see \cite[Chapter VI]{M} for details). Take
two sets of variables $X=(x_1,x_2,\dots)$ and $Y=(y_1,y_2,\dots)$ and a symmetric function
$f\in\Lambda$. Let $(X,Y)$ be the union of sets of variables $X$ and $Y$. Then we can view
$f(X,Y)\in\Lambda_{(X,Y)}$ as a symmetric function in $x_i$ and also a symmetric function in $y_i$,
more precisely, $f(X,Y)$ is a sum of products of symmetric functions of $x_i$ and symmetric
functions of $y_i$. More formally, this operation defines a comultiplication
$\Delta:\Lambda\to\Lambda\otimes\Lambda$, which turns $\Lambda$ into a bi--algebra (see e.g.\
\cite{Zel}).

Skew Macdonald symmetric functions $P_{\lambda/\mu}$, $Q_{\lambda/\mu}$ are defined as the
coefficients in the expansions
\begin{equation}
\label{eq_skew}
 P_\lambda(X,Y)=\sum_\mu P_\lambda(X) P_{\lambda/\mu}(Y), \qquad \textrm{and}\qquad Q_\lambda(X,Y)=\sum_\mu Q_\lambda(X) Q_{\lambda/\mu}(Y).
\end{equation}
Both $P_{\lambda/\mu}$ and $Q_{\lambda/\mu}$ are homogeneous symmetric functions of degree
$|\lambda|-|\mu|$, moreover $P_{\lambda/\mu}=Q_{\lambda/\mu}=0$ unless $\mu\subset\lambda$ (which
means that $\mu_i\le\lambda_i$ for $i=1,2,\dots$).

\subsection{Topology}

Given a $\mathbb Z_{\ge 0}$--graded algebra $A$, its \emph{topological completion} $\overline A$ is
defined as the algebra of all formal series
\begin{equation}
\label{eq_completion} a=\sum_{k=0}^{\infty} a_k, \quad a_k\in A,\, \deg(a_k)=k.
\end{equation}
For any element $a\in \overline A$, its \emph{lower degree} $\ldeg(a)$ is defined as a maximal $K$
such that $a_k=0$ in \eqref{eq_completion} for all $k<K$. We equip $\overline A$ with a
\emph{graded} topology in which a sequence $b^n$ converges to $b\in\overline A$ if and only if
$$
 \lim_{n\to\infty} \ldeg(b-b_n)=+\infty.
$$
In this topology $A$ is a dense subalgebra of $\overline A$. A \emph{completed graded algebra} is
defined as a topological completion of some $\mathbb Z_{\ge 0}$--graded algebra.

Given two graded algebras $A$ and $B$, we equip their tensor product $A\otimes B$ with a unique
grading such that
$$
 \deg(a\otimes b)=\deg(a)+\deg(b)
$$
for any homogeneous $a\in A$, $b\in B$. Given two completed graded algebras $\overline A$ and
$\overline B$, their tensor product is defined as
$$
 \overline A \otimes \overline B = \overline{A\otimes B}.
$$
Given a completed graded algebra $\overline A$ and a graded algebra $B$, their tensor product is
defined as
$$
 \overline A \otimes B = \overline{A \otimes_1 B},
$$
where $A \otimes_1 B$ is the tensor product $A\otimes B$ equipped with a unique grading such that
$$
 \deg(a\otimes b)=\deg(a),
$$
for any homogeneous $a\in A$, $b\in B$. $B\otimes\overline A$ is defined similarly (and is
canonically isomorphic to $\overline A\otimes B$).

Note the difference between $\overline{A\otimes B}$ and $\overline{A \otimes_1 B}$. For instance,
if $A=\mathbb C[x]$ and $B=\mathbb C[y]$ with grading by degree of the polynomials, then
$(1-xy)^{-1}=\sum_{n=0}^{\infty} x^n y^n$ belongs both to $\overline{A\otimes B}$ and $\overline{A
\otimes_1 B}$. Moreover, $(1-x)^{-1}=\sum_{n=0}^{\infty} x^n$ also belongs to both tensor products.
However, $(1-y)^{-1}=\sum_{n=0}^{\infty} y^n$ belongs only to $\overline{A\otimes B}$, but not to
$\overline{A\otimes_1 B}$.

Now take three algebras $\mathcal A$, $\mathcal B$, $\mathcal C$ such that $\mathcal
C\simeq\Lambda$, while $\mathcal A$ and $\mathcal B$ are either graded or completed graded
algebras.
\begin{definition} The Macdonald pairing $\l \cdot, \cdot \r_{\mathcal C}$ is a unique (continuous) bilinear map
$$
 ({\mathcal A}\otimes {\mathcal C}) \times ( {\mathcal C} \otimes {\mathcal B}) \to \mathcal A
 \otimes \mathcal B,
$$
 such that
$$
 \l a\otimes c_1, c_2\otimes b \r_{\mathcal C} = \l c_1,c_2 \r\, a\otimes b.
$$
\end{definition}
{\bf Remark.} When $\mathcal C$ is the algebra $\Lambda_X$ of symmetric functions in variables
$X=(x_1,x_2,\dots)$, we will also use the notation $\l \cdot, \cdot \r_{X}$ for $\l \cdot, \cdot
\r_{\mathcal C}$.

\smallskip

Note that our definitions imply an alternative definition of skew Macdonald symmetric functions:
$$
 P_{\lambda/\mu}(X) = \l P_\lambda(X,Y), Q_\mu(Y) \r_Y.
$$

The following property of the Macdonald pairing is crucial for us.
\begin{proposition}
\label{proposition_evaluation_of_scalar_product}
 Let $p_k$ be the Newton power sums in $\Lambda$
 and let $a_k$, $b_k$ be two sequences of elements of graded algebras $A$ and $B$ with
 $\lim\limits_{k\to\infty} \ldeg(a_k)=\lim\limits_{k\to\infty} \ldeg(b_k)=\infty$, so that
\begin{equation}
\label{eq_summability_condition}
 \sum_{k=1}^\infty \frac{a_k p_k}{k}\in \overline {A} \otimes \Lambda, \quad  \sum_{k=1}^\infty \frac{b_k p_k}{k}\in \overline {B} \otimes
 \Lambda.
\end{equation}
 Then
$$
  \left\langle \exp\left(\sum_{k=1}^\infty \frac{a_k p_k}{k}\right),\,  \exp\left(\sum_{k=1}^\infty
  \frac{p_kb_k}{k}\right) \right\rangle_{\Lambda} = \exp\left(\sum_{k=1}^\infty\left( \frac{1-q^k}{1-t^k}\cdot \frac{a_k
  b_k}{k}\right)\right),
$$
where the right--hand side is an element of $\overline{A\otimes B}$.
\end{proposition}
{\bf Remark.} The condition \eqref{eq_summability_condition} is satisfied, in particular, if $a_k$,
$b_k$ are two sequences of \emph{homogeneous} elements of graded algebras $A$ and $B$,
respectively, such
 that $\deg(a_k)=\deg(b_k)=k$.
\begin{proof}[Proof of Proposition \ref{proposition_evaluation_of_scalar_product}]
Take three copies $\Lambda_X$, $\Lambda_Y$, $\Lambda_Z$ of the algebra of symmetric functions.
Definitions imply that
\begin{equation}
\label{eq_x11}
 \l \sum_{\lambda\in\Y} P_\lambda(X) Q_\lambda(Y), \sum_{\lambda\in\Y} P_\lambda(Y) Q_\lambda(Z) \r_{\Lambda_Y} =
 \sum_{\lambda\in\Y} P_\lambda(X) Q_\lambda(Z).
\end{equation}
The Cauchy--type identity for Macdonald symmetric functions (see \cite[Chapter VI, (2.7)]{M})
yields
\begin{equation}
\label{eq_Mac_cauchy}
 \sum_{\lambda\in\Y} P_\lambda(X) Q_\lambda(Y) =\exp \left(\sum_{k=1}^\infty \left(\frac{1-t^k}{1-q^k}
 \cdot \frac{p_k(X) p_k(Y)}{k} \right)\right)
\end{equation}
and similarly for the sets of variables $Y$ and $Z$. Then \eqref{eq_x11} implies that
\begin{align}
\label{eq_x12}
 &\left\langle \exp \left(\sum_{k=1}^\infty \left(\frac{1-t^k}{1-q^k}
 \cdot \frac{p_k(X) p_k(Y)}{k} \right)\right),\exp \left(\sum_{k=1}^\infty \left(\frac{1-t^k}{1-q^k}
 \cdot \frac{p_k(Y) p_k(Z)}{k} \right)\right) \right\rangle_{\Lambda_Y}\\
 \nonumber&=
 \exp \left(\sum_{k=1}^\infty \left(\frac{1-t^k}{1-q^k}
 \cdot \frac{p_k(X) p_k(Z)}{k} \right)\right).
\end{align}
Now let $\varphi_{X,A}$ be a (continuous, algebra--) homomorphism from $\overline{\Lambda_X}$ to
$\overline {A}$ such that:
$$
 \varphi_{X,A}: \overline{\Lambda_X} \to \overline {A},\quad \varphi_{X,A}(p_k(X))=\frac{1-q^k}{1-t^k}
 a_k.
$$
Also let $\varphi_{Z,B}$ be a (continuous, algebra--) homomorphism from $\overline{\Lambda_Z}$ to
$\overline {B}$ such that:
$$
 \varphi_{Z,B}: \overline{\Lambda_Z} \to \overline {B},\quad \varphi_{Z,B}(p_k(Z))=\frac{1-q^k}{1-t^k}
 b_k.
$$
Applying $\varphi_{X,A}$ and $\varphi_{Z,B}$ to the identity \eqref{eq_x12} we are done.
\end{proof}

\subsection{Formal measures}

\label{section_formal_measures}

Let $\mathcal N$ be a countable set and let $\mathcal A$ be a completed graded algebra.

\begin{definition}
 A formal probability measure $P$ on $\mathcal N$ taking values in $\mathcal A$ is a map
 $
  P: \mathcal N\to \mathcal A,
 $
 such that
 $$
  \sum_{\eta\in\mathcal N} P(\eta)=1.
 $$
\end{definition}

The following procedure constructs a conventional probability measure on $\mathcal N$ from a
formal one. Take a graded algebra $A$. A \emph{specialization} $\rho$ is an (algebra--)
homomorphism $
 \rho: A\to \mathbb C.
$

 An arbitrary element  $g$  of $\overline{A}$ can be uniquely represented as
$$
 \sum_{k=0}^\infty g_k,\quad \quad \deg(g_i)=i,\quad i\in\mathbb Z_{\ge 0}.
$$
Define the $\rho$--seminorm on (a subset of) $\overline{A}$ through
$$
 \| g\|_\rho= \sum_{k=0}^\infty |\rho(g_k)|.
$$
Let $A_\rho\subset \overline{A}$ denote the subset of elements with finite $\rho$--seminorm in
$\overline{A}$. Clearly, $A_\rho$ is a subalgebra of $A$ and $\rho$ is uniquely extended to a
continuous (in $\rho$--seminorm) homomorphism from $A_\rho$ to $\mathbb C$, that we denote by the
same letter $\rho$.

\begin{definition}
 Let $P$ be a formal probability measure on $\mathcal N$ taking values in $\overline A$. A
 specialization $\rho$ of $A$ is called $P$--\emph{positive}, if for any $\eta\in\mathcal N$,
 $P(\eta)\in A_\rho$, $\rho(P(\eta))\ge 0$, and also the series
$\sum_{\eta\in\mathcal N} P(\eta)$ converges (to $1$) in $\rho$--seminorm.
\end{definition}
Clearly, any $P$--positive specialization $\rho$ defines a probability measure on $\mathcal N$
through the formula
$$
 {\rm Prob}(\eta)=\rho(P(\eta)).
$$

\section{Formal Macdonald processes and observables}

\label{Section_Formal_Macdonald}

\subsection{Formal Macdonald process}\label{formmacproc}

For two (finite or countable) sets of variables $X=(x_1,x_2,\dots)$ and $Y=(y_1,y_2,\dots)$, define
$\Pi(X;Y)$ through
$$
 \Pi(X;Y) =\prod_{i,j} \frac{(t x_i y_j;q)_{\infty}}{(x_i y_j;q)_{\infty}},
$$
where we used the $q$--Pochmamer symbol notation:
$$
 (a;q)_\infty=(1-a)(1-aq)(1-aq^2)\cdots.
$$
If the sets $X$ and $Y$ are countable, then $\Pi(X;Y)$ is an element of
$\overline{\Lambda_X}\otimes \overline{\Lambda_Y}$; one easily checks that it is related to the
generators $p_k(X)$, $p_k(Y)$ through the following formula:
$$
 \Pi(X;Y)=\exp\left(\sum_{k=1}^{\infty} \frac{1-t^k}{1-q^k} \cdot \frac{p_k(X) p_k(Y)}{k}\right).
$$
Note that $\Pi(X;Y)$ can be inverted, and $\Pi(X;Y)^{-1}$ is also an element of
$\overline{\Lambda_X}\otimes \overline{\Lambda_Y}$.

$\Pi(X;Y)$ can be also related to the Macdonald symmetric functions (see \cite[Chapter VI,
(2.7)]{M} and \eqref{eq_Mac_cauchy} above):
\begin{equation}
\label{eq_Cauchy_for_Mac}
 \Pi(X;Y)=\sum_{\lambda\in\Y} P_\lambda(X) Q_\lambda(Y).
\end{equation}
The definition also implies that for more than two sets of variables we have
\begin{equation}
\label{eq_Pi_product}
 \Pi(X^1,\dots X^k; Y^1,\dots,Y^m)=\prod_{i=1}^k \prod_{j=1}^m \Pi(X^i;Y^j).
\end{equation}

\begin{definition}
Take two countable sets of variables $A$ and $B$. The formal Macdonald measure $\MF_{A,B}$ is a
formal probability measure on $\Y$ taking values in
$\overline{\Lambda_A}\otimes\overline{\Lambda_B}$ such that
$$
 \MF_{A,B}(\lambda) = \frac{ P_\lambda(A) Q_\lambda(B)} {\Pi(A;B)},\quad \lambda\in \Y.
$$
\end{definition}

\begin{definition}
\label{Definition_mac_pro} Fix integer $N>0$ and $2N$ sets of variables
$\mathbf{A}=(A^1,\dots,A^N)$, $\mathbf B=(B^1,\dots,B^N)$. The formal Macdonald process
$\MPF_{N,\mathbf{A},\mathbf{B}}$ is a formal probability measure on $\Y^N$ taking values in
$$
 \overline{\Lambda_{A^1}}\otimes \dots \otimes  \overline{\Lambda_{A^N}}\otimes
 \overline{\Lambda_{B^1}}\otimes \dots  \otimes \overline{\Lambda_{B^N}}
$$
such that
\begin{equation}
\label{eq_definition_of_Mac_pro_classical}
 \MPF_{N,\mathbf{A},\mathbf{B}}(\lambda^1,\dots,\lambda^N) =
 \frac{P_{\lambda^1}(A^1) \Psi_{\lambda^2,\lambda^{1}}(A^2;B^{1})
\Psi_{\lambda^3,\lambda^{2}}(A^3;B^{2}) \cdots  \Psi_{\lambda^N,\lambda^{N-1}}(A^N;B^{N-1})
Q_{\lambda^N}(B^N)}{ \prod_{1\le \alpha\le \beta \le N} \Pi(A^\alpha; B^\beta)}
\end{equation}
where
$$
\Psi_{\lambda,\mu}(A;B)=\sum_{\nu\in\Y} P_{\lambda/\nu}(A) Q_{\mu/\nu}(B).
$$
\end{definition}
{\bf Remark.} Definition \ref{Definition_mac_pro} is a generalization of the definition of the
Schur process of \cite{OR_Schur} that arises when $q=t$.

\smallskip

The fact that the formal Macdonald measure is a formal probability measure on $\Y$ is immediate from
\eqref{eq_Cauchy_for_Mac}. For the formal Macdonald process this fact is a bit more involved to see, and so
we provide a proof in Proposition \ref{Prop_mass_one}. In what follows we will actually use an
equivalent definition of the Macdonald process which we now present.
\begin{proposition}
\label{Prop_Mac_pro_alternative} In the settings of Definition \ref{Definition_mac_pro} we have
\begin{eqnarray}
\label{eq_definition_of_Mac_pro}
\nonumber \MPF_{N,\mathbf{A},\mathbf{B}}(\lambda^1,\dots,\lambda^N) &=& \frac{1}{ \prod_{1\le \alpha\le \beta \le N} \Pi(A^i; B^j)}\\
&&\times
 Q_{\lambda^N}(B^N) \l P_{\lambda^N}(A^N,Y^{N-1}), Q_{\lambda^{N-1}}(Y^{N-1},B^{N-1}) \r_{Y^{N-1}} \\
\nonumber&& \times \l P_{\lambda^{N-1}}(A^{N-1},Y^{N-2}), Q_{\lambda^{N-2}}(Y^{N-2},B^{N-2}) \r_{Y^{N-2}}\\
\nonumber&&\times  \cdots \times  \l P_{\lambda^2}(A^2,Y^1), Q_{\lambda^1}(Y^1,B^1) \r_{Y^1} P_{\lambda^1}(A^1).
\end{eqnarray}
\end{proposition}
\begin{proof}
 This follows from the identity (we use \eqref{eq_skew})
\begin{eqnarray}
\label{eq_scalar_product_evaluation}
\nonumber&&\l P_{\lambda^k}(A^k,Y^{k-1}),
Q_{\lambda^{k-1}}(Y^{k-1},B^{k-1}) \r_{Y^{k-1}}\\
&& = \l \sum_{\nu\in\Y} P_{\lambda^k/\nu}(A^k)
P_\nu(Y^{k-1}), \sum_{\nu'} Q_{\lambda^{k-1}/\nu'}(B^{k-1}) Q_{\nu'}(Y^{k-1}) \r_{Y^{k-1}}
\\
\nonumber&&=\sum_{\nu\in\Y} P_{\lambda^k/\nu}(A^k)
Q_{\lambda^{k-1}/\nu}(B^{k-1})=\Psi_{\lambda^k,\lambda^{k-1}}(A^k;B^{k-1}).\qedhere
\end{eqnarray}
\end{proof}

\begin{proposition} \label{Prop_mass_one}We have
$$
 \sum_{\lambda^1,\dots,\lambda^N\in\Y} \MPF_{N,\mathbf{A},\mathbf{B}}(\lambda^1,\dots,\lambda^N) =1.
$$
\end{proposition}
\begin{proof}
 Summing \eqref{eq_definition_of_Mac_pro} over $\lambda^1,\dots,\lambda^N$ and using
 \eqref{eq_Cauchy_for_Mac} we get
\begin{align*}
 \frac{1}{ \prod_{\alpha\le \beta} \Pi(A^\alpha; B^\beta)}\, \dot
  \l \Pi\left(B^N;A^N,Y^{N-1}\right), \l \Pi\left(Y^{N-1},B^{N-1}; A^{N-1},Y^{N-2}\right),  \\ \cdots
 \l \Pi( B^2,Y^2;A^2,Y^1), \Pi(Y^1,B^1;A^1) \r_{Y^1}\dots \r_{Y^{N-1}}.
\end{align*}
It remains to use Proposition \ref{proposition_evaluation_of_scalar_product} in the form
$$
 \l \Pi(U;Y^k), \Pi(Y^k;V) \r_{Y^k} = \Pi(U;V)
$$
for $k=1,\dots, N-1$ and appropriate sets of variables $U$ and $V$, as well as
\eqref{eq_Pi_product}.
\end{proof}

Two simple, yet important properties of formal Macdonald processes are summarized in the following
propositions.

\begin{proposition}
\label{prop_Mac_pro_empty} In the notations of Definition \ref{Definition_mac_pro}, let $\phi^i_0$
denote the map
$$
 \phi^i_0: \Lambda_{A^{i+1}}\otimes \Lambda_{B^i}\to\mathbb C, \quad \quad \phi^i_0(f\otimes g)=\phi_0(f)\phi_0(g),
$$
where $\phi_0$ is the constant term map, as above. Further, let $\mathbf{A}^{(j)}$,
$\mathbf{B}^{(j)}$ denote the sets of variables $\mathbf{A}\setminus A^j$ and $\mathbf{B}\setminus
B^j$, respectively. For $1\le i\le N-1$ consider the formal measure
$$M^i=\phi^i_0\left(
\MPF_{N,\mathbf{A},\mathbf{B}}(\lambda^1,\dots,\lambda^N) \right).$$ Then for all sequences
$(\lambda^1,\dots,\lambda^N)\in\Y^N$ in the support of $M^i$, we have $\lambda^{i}=\lambda^{i+1}$.
Furthermore, the restriction of $M^i$ to
$(\lambda^1,\dots,\lambda^{i-1},\lambda^{i+1},\dots,\lambda^N)$ is the formal Macdonald process
$\MPF_{N-1,\mathbf{A}^{(i+1)},\mathbf{B}^{(i)}}$.
\end{proposition}
\begin{proof} This readily follows from the identities
$$
 \phi^i_0 \left(\Psi_{\lambda,\mu}(A^{i+1},B^i)\right) = \delta_{\lambda,\mu},\qquad \textrm{and} \qquad
 \phi^i_0 \left(\Pi(A^{i+1};B^j) \right)= \phi^i_0 \left(\Pi(A^j;B^i) \right) = 1.\qedhere
$$
\end{proof}

\begin{proposition}
\label{prop_Mac_pro_restriction} In the notations of Definition \ref{Definition_mac_pro}, let
$\mathbf{A}^{i \cup i+1}$ denote $N-1$ sets of variables
 $\{A_1, A_2, \dots , A_{i-1}, (A_i,
A_{i+1}), A_{i+2}, \dots, A_N\}$, i.e.\ we unite $A^i$ and $A^{i+1}$ into a single set. Similarly
define $\mathbf{B}^{i\cup i+1}$. Then the restriction of
$\MPF_{N,\mathbf{A},\mathbf{B}}(\lambda^1,\dots,\lambda^N)$ to
$(\lambda^1,\dots,\lambda^{i-1},\lambda^{i+1},\dots,\lambda^N)$, $1\le i \le N$,is the formal
Macdonald measure $\MPF_{N-1,\mathbf{A}^{i \cup i+1},\mathbf{B}^{i-1 \cup i}}$.
\end{proposition}
\begin{proof} For $1<i<N$ this follows from the following identity, which is a combination of
\cite[Exercise 6, Section 7, Chapter VI]{M} and \eqref{eq_skew}:
\begin{align*}
& \sum_{\lambda^i\in\Y} \Psi_{\lambda^{i+1},\lambda^i}(A^{i+1};B^i) \Psi_{\lambda^{i},\lambda^{i-1}}(A^{i};B^{i-1})\\
& = \sum_{\lambda^i,\mu,\nu\in\Y} P_{\lambda^{i+1}/\mu}(A^{i+1})Q_{\lambda^i/\mu}(B^i)
 P_{\lambda^i/\nu}(A^i) Q_{\lambda^{i-1}/\nu}(B^{i-1})\\
& = \Pi(A^i;B^i)\sum_{\kappa,\mu,\nu\in\Y} P_{\lambda^{i+1}/\mu}(A^{i+1})P_{
 \mu/\kappa}(A^i)  Q_{\nu/\kappa}(B^i) Q_{\lambda^{i-1}/\nu}(B^{i-1})\\
& = \Pi(A^i;B^i) \sum_{\kappa\in\Y} P_{\lambda^{i+1}/\mu}(A^{i+1},A^i) Q_{\lambda^{i-1}/\kappa}(B^i,B^{i-1})\\
& = \Pi(A^i;B^i) \Psi_{\lambda^{i+1},\lambda^{i-1}}\big((A^{i+1},A^i);(B^i,B^{i-1})\big).
\end{align*}
For $i=1$ and $i=N$ the argument is similar.
\end{proof}

\subsection{Single level observables}
For two sets of variables $X=(x_1,x_2,\dots)$ and $Y=(y_1,y_2,\dots)$, let $\HH(X;Y)$ be the
Hall-Littlewood (i.e., $q=0$) specialization of $\Pi$:
\begin{equation}
\label{eq_H_function}
 \HH(X;Y) =\prod_{i,j} \frac{1-t x_i y_j}{1-x_i y_j}=\exp\left(\sum_{k=1}^{\infty} (1-t^k)
 \frac{p_k(X) p_k(Y)}{k}\right).
\end{equation}

\begin{definition}
\label{Definition_observable}
 The function $\O_r: \Y\to \mathbb C$ is defined through
$$
 \O_r (\lambda)= \lim_{N\to\infty} e_r(q^{-\lambda_1},q^{-\lambda_2}t,\dots, q^{-\lambda_N}
 t^{N-1}),\quad r\ge 1,
$$
where $e_r$ is the elementary symmetric polynomial and $\O_0 (\lambda)=1$.
\end{definition}
For example,
$$
\O_1(\lambda)=\lim_{N\to\infty} \sum_{i=1}^N q^{-\lambda_i} t^{i-1}=\sum_{i=1}^{\ell(\lambda)}
q^{-\lambda_i} t^{i-1} + \frac{t^{\ell(\lambda)}}{1-t}.
$$

Our first result is the computation of the expectation of the observables $\O_r(\lambda)$ with
respect to a formal Macdonald measure.
\begin{proposition} For two sets of variables $X$ and $Y$ we have
\label{proposition_restated_computation_from_BigMac}
\begin{equation}
\label{eq_main_lemma}
 \sum_{\lambda\in\Y} \O_r(\lambda) \MF_{X,Y}(\lambda)
% = \frac{1}{\Pi(X;Y)}\sum_{\lambda\in\Y} \O_r(\lambda)P_\lambda(X)
% Q_\lambda(Y)
= \frac{1}{(2\pi \i)^r r!} \oint\limits_{|w_1|=1}\dots\oint\limits_{|w_r|=1} \det \left[ \frac{1}{
 w_k-tw_\ell}\right]
 \prod_{j=1}^r \HH\big(w_j;X\big)\, \HH\big((qw_j)^{-1}; Y\big) dw_j.
\end{equation}
\end{proposition}
Let us explain how Proposition \ref{proposition_restated_computation_from_BigMac} should be
understood. Clearly, the left side of \eqref{eq_main_lemma} is an element of
$\overline{\Lambda_X}\otimes\overline{\Lambda_Y}$. Turning to the right side, by definition, for a set of variables $X$ and a single variable $u$, we
have
$$
\HH(u;X)=\exp\left(\sum_{k=1}^{\infty} (1-t^k)\frac{u^k p_k(X)}{k}\right).
$$
Therefore, the integrand on the right--hand side pf \eqref{eq_main_lemma} can be (uniquely) written as a sum
\begin{equation}
\label{eq_x13}
 \sum_{k=0}^{\infty} f_k(w_1,\dots,w_r) g_k,
\end{equation}
where $f_k$ is a certain function of $w_1,\dots,w_r$ and $g_k$ is an element of
$\Lambda_X\otimes\Lambda_Y$ of degree $k$. When we integrate \eqref{eq_x13} termwise (with $w_j$
integrated over the unit circle $|w_j|=1$), we are left with an element of
$\overline{\Lambda_X}\otimes\overline{\Lambda_Y}$. Now Proposition
\ref{proposition_restated_computation_from_BigMac} claims that this element is the same as the one
in the left side of \eqref{eq_main_lemma}.

The integrals over $w_j$ can be understood analytically (as complex integrals over contours) or,
equivalently, they have a purely algebraic meaning. Indeed, expand $\det(\frac{1}{
 w_k-tw_\ell})$ in the integrand in a power series using (recall that $0<t<1$)
\begin{equation}
\label{eq_x18}
 \frac{1}{w_k-t w_\ell} = \frac{1}{w_k} \cdot \frac{1}{1-t w_\ell/w_k} =\frac{1}{w_k} \sum_{i=0}^{\infty} \frac{
 t^i (w_\ell)^i}{(w_k)^{i}}.
\end{equation}
 Note that multiplication of series \eqref{eq_x18} for various indices $k$ and $\ell$ might involve
 summing geometric progressions with ratio $t$. After this procedure
 the functions $f_k$ in \eqref{eq_x13} become power series (in $w_i$ and $w_i^{-1}$). The contour
integral of such power series over the unit circle is $(2\pi\i)^r$ times the coefficient of
$(w_1\cdots w_r)^{-1}$.

\smallskip

{\bf Remark 1.} Both left and right sides of \eqref{eq_main_lemma} are symmetric under
interchanging $X$ and $Y$ (a change of integration variables is needed to see the symmetry in the
right side).

{\bf Remark 2.} The formula is also valid for $e_0=1$ if we understand the empty integral as $1$.

{\bf Remark 3.} If the integral is understood analytically, then the contours of integration can be
chosen along the circles $|w_j|=R>0$, $j=1,\dots,r$. The actual value of $R$ does not matter, as we
can deform all the contours together without changing the value of the integral.

{\bf Remark 4.} An integral representation similar to \eqref{eq_main_lemma} can be found in
\cite[Section 9]{Shi} and \cite[Proposition 3.6]{FHHSY} under the name of Heisenberg
Representation of the Macdonald Difference Operators.

\smallskip

The proof of Proposition \ref{proposition_restated_computation_from_BigMac} relies upon the
following lemma.

\begin{lemma}
\label{lemma_from_BigMac}
 Take two sets of $N$ complex numbers $X=\{x_i\}_{i=1}^N$ and $Y=\{y_i\}_{i=1}^N$ such that $|x_i y_j|<1$, $1\le
i,j\le N$. Assume that there exist $r$ closed complex contours $\gamma_1$,\dots $\gamma_r$, such
that the integral
\begin{equation}
\label{eq_x14} \frac{1}{(2\pi \i)^r r!} \oint_{\gamma_1}\dots\oint_{\gamma_r} \det \left[ \frac{1}{
 w_k-tw_\ell}\right]
 \prod_{j=1}^r \HH\big(w_j;X\big)\, \HH\big((qw_j)^{-1}; Y\big) dw_j
\end{equation}
is equal to the sum of the residues of the integrand at $w_j=(x_i)^{-1}$ for $j=1,\dots,r$,
$i=1,\dots,N$. Then the integral \eqref{eq_x14} also equals
$$
\sum\limits_{\lambda\in\Y: \ell(\lambda)\le N}
 e_r(q^{-\lambda_1}t^{0},\dots,q^{-\lambda_N} t^{N-1})\dfrac{P_\lambda(x_1,\dots,x_N) Q_\lambda(y_1,\dots,y_N)}{\Pi(x_1,\dots,x_N;y_1,\dots,y_N)}.
$$
\end{lemma}
\begin{proof} This fact can be found in \cite[Remark 2.2.11]{BigMac}.
 The proof is based on the application of the $r$th Macdonald difference operator in variables $X$ (see \cite[Chapter VI]{M})
 to the identity
 $$
   \sum\limits_{\lambda\in\Y: \ell(\lambda)\le N}
P_\lambda(X) Q_\lambda(Y) = \Pi(X;Y).
 $$
 See also Section \ref{Section_ascending_obs} for more details.
\end{proof}

\begin{proof}[Proof of Proposition \ref{proposition_restated_computation_from_BigMac}]

Fix three reals $0<R_1<R_2<R_3$ such that $t R_3<R_1$. Take $N$ complex numbers
$X=\{x_i\}_{i=1}^N$ and $N$ complex numbers $Y=\{y_i\}_{i=1}^N$ such that that
$R_2<|x_j|^{-1}<R_3$ and $|y_j|\ll R_1$ for all $i$. In what follows we assume that $x_i$'s are
distinct, but all the formulas are readily extended to the case of equal $x_i$'s by  continuity.

Consider the integral
\begin{equation}
\label{eq_x15} \frac{1}{(2\pi \i)^r r!} \oint\dots\oint \det \left[ \frac{1}{
 w_k-tw_\ell}\right]
 \prod_{j=1}^r \HH\big(w_j;X\big)\, \HH\big((qw_j)^{-1}; Y\big) dw_j
\end{equation}
 with each $w_j$ being integrated over the union of circles $|w_j|=R_1$ and $|w_j|=R_3$ with the
integral over $R_1$ being positively orientated and over $R_3$ begin negatively oriented. The restrictions
on the variables imply that the integral is equal to the sum of the residues at $w_i=(x_j)^{-1}$ for $i=1,\ldots, r$, $j=1,\ldots, N$.
Thus we can apply Lemma \ref{lemma_from_BigMac} to see that the above integral equals
\begin{equation}
\label{eq_x16}
 \sum_{\lambda\in\Y^N} e_r(q^{-\lambda_1},q^{-\lambda_2}t, \dots, q^{-\lambda_N} t^{N-1}) \frac{P_\lambda(X)
 Q_\lambda(Y)}{\Pi(X;Y)}.
\end{equation}
Our aim is to convert the analytic identity \eqref{eq_x15}\,$=$\,\eqref{eq_x16} into the formal
identity in completed graded algebras which constitutes Proposition
\ref{proposition_restated_computation_from_BigMac}.

Note that \eqref{eq_x16} has a unique expansion as a (symmetric) power series in $x_j$'s and $y_j$'s. Any such symmetric power series can be written as a power series in
$p_k(X)$, $p_k(Y)$. As $N\to\infty$, each coefficient of the expansion for \eqref{eq_x16} converges to those of the left--hand side of \eqref{eq_main_lemma}.
Therefore, it remains to show similarly that the coefficients of the expansion in power series in $p_k(X)$, $p_k(Y)$ of \eqref{eq_x15} converge to the corresponding ones on the right side of \eqref{eq_main_lemma}. The rest of the proof is devoted to showing this.

\smallskip

The first step is to replace the portion of the contour of integration in which $w_j$ is integrated
along the circle of radius $R_3$ by a circle of radius $R_4 \gg 1$. We claim that the integral does not
change value under this transformation. To see this fact, recall that before the deformation, the integral is equal to the sum
of the residues of the integrand at points $w_i=(x_j)^{-1}$, $i=1,\dots,r$, $j=1,\dots,N$. Let us
also compute the integral (via residues) after the deformation of the contours and show it matches. First, we integrate
over $w_1$, getting the residues from the $N+2r-2$ choices of poles of the integrand at $w_1=(x_j)^{-1}$, $j=1,\dots,N$, and also at $w_1=t w_i$,
$i=2,\dots,r$ and $w_1=t^{-1} w_i$, $i=2,\dots,r$. For each choice of pole, we further integrate over $w_2$, picking residues in a similar manner, and so on upto $w_r$. From this we see that the integral is expanded into a sum of residues of the integrand in \eqref{eq_x15} over points of the form
\begin{equation}
\label{eq_x31}
 w_1=(x_{j_1})^{-1} t^{p_1},\quad w_2=(x_{j_2})^{-1} t^{p_2}, \ldots, w_{r} = (x_{j_r})^{-1} t^{p_r}
\end{equation}
where the summation is restricted to a certain subset (which we will determine in a moment) of $j_1,j_2,\ldots, j_r\in \{1,\ldots, N\}$ and $p_1,p_2,\ldots, p_r\in \Z$.

In order to determine which subset of points of the form of \eqref{eq_x31} should be summed over, note the following properties: If at least one of the pairs coincide, i.e., $(j_m,p_m)=(j_n,p_n)$ for $m\neq n$, then the residue is zero, since the integrand has no singularity at such a point. This is because
the Cauchy determinant (see e.g.\ \cite{Kr})
\begin{equation}
\label{eq_Cauchy_det} \det \left[ \frac{1}{
 w_k-tw_\ell}\right]_{k,\ell=1}^r=\frac{t^{r(r-1)/2}}{(1-t)^r w_1\dots w_r } \prod_{k\ne
 \ell} \frac{w_k-w_\ell}{w_k-tw_\ell}
\end{equation}
 vanishes when some of the variables coincide. Further, all $p_i$ should be non--positive. Indeed,
 no point of the kind $(x_j)^{-1} t^{-k}$, $k>0$ is inside our contours. We may further observe that the summation of residue need only be taken over points in \eqref{eq_x31}
 which are a union of strings of the form
 $$
  w_{i_1}=(x_j)^{-1},\quad w_{i_2}=(x_j)^{-1} t^{-1}, \quad \dots,\quad w_{i_m}= (x_j)^{-1} t^{1-m},
 $$
 (i.e., each string has the above form, but with possibly different length $m$, possibly different $j$ and disjoint variables $i_1,i_2,\ldots,i_m$).
 Note that if the length of any given string (i.e., $m$ in the above formula) is at least $2$, then the
 residue at such point vanishes. Indeed, the pole arising from the
determinant in the integrand cancels out with corresponding zero of $\HH(w;x_j)$. On the other hand,
if all the strings are of length $1$, then we get the same sum as was before the deformation of the
contours --- thus proving our claim.

\smallskip

The integral in \eqref{eq_x15} with $R_3$ replaced now by $R_4$ can be written as a sum of $2^r$ contour--integrals over circular contours with some
variables integrated over the circle of radius $R_1$ and others over the circle of radius $R_4 \gg 1$. Our
 aim is to analyze each term and ultimately show that as $N\to\infty$ only the term with all integrations
 over the $R_1$ circle survives. Since the integrand is symmetric in $w_j$, it is enough to consider the
 case when
$|w_1|=|w_2|=\dots=|w_m|=R_4$ and $|w_{m+1}|=\dots=|w_r|=R_1$, i.e., the integral
\begin{multline}
\label{eq_x24}
 \frac{1}{(2\pi \i)^r r!} \oint\limits_{|w_{m+1}|=R_1}\dots\oint\limits_{|w_{r}|=R_1} \prod_{j={m+1}}^r \HH\big(w_j;X\big)\, \HH\big((qw_j)^{-1}; Y\big)
\\ \times
   \oint\limits_{|w_1|=R_4}\dots\oint\limits_{|w_m|=R_4}  \det \left[ \frac{1}{
 w_k-tw_\ell}\right]_{k,\ell=1}^r
 \prod_{j=1}^m \HH\big(w_j;X\big)\, \HH\big((qw_j)^{-1}; Y\big) \prod_{j=1}^{r} dw_j.
\end{multline}

 Using the Cauchy
determinant formula \eqref{eq_Cauchy_det} we write:
\begin{eqnarray*}
\det \left[ \frac{1}{w_k-tw_\ell}\right]_{k,\ell=1}^r &=&\frac{t^{r(r-1)/2}}{(1-t)^r w_1\dots w_r } \prod_{k\ne
 \ell} \frac{w_k-w_\ell}{w_k-tw_\ell}\\
 &=& \frac{1}{(1-t)^r w_1\dots w_r } \prod_{k<\ell
 } \frac{1-w_\ell/w_k}{1-t^{-1}w_\ell/w_k} \prod_{k<\ell} \frac{1-w_\ell/w_k}{1-tw_\ell/w_k}.
\end{eqnarray*}
Note that $|w_\ell/w_k|$ equals either $1$ or $R_1/R_4\ll 1$ on our contours for $k<\ell$.
Therefore,
$$
\det \left[ \frac{1}{
 w_k-tw_\ell}\right]_{k,\ell=1}^r= \det \left[ \frac{1}{
 w_k-tw_\ell}\right]_{k,\ell=m+1}^r \det \left[ \frac{1}{
 w_k-tw_\ell}\right]_{k,\ell=1}^m \Big(1+O\big((R_4)^{-1}\big)\Big) t^{m(r-m)},
$$
where the remainder $O(\cdot)$ is uniform over integration variables $w_j$ on our contours.
 For $j=1,\dots,m$ note that
$$
 \HH(w_j;X)=\prod_{k=1}^N\frac{1-t w_j x_k}{1-w_j x_k}=t^N \prod_{k=1}^N \frac{1-t^{-1} (w_j)^{-1}
 (x_k)^{-1}}{1-(w_j)^{-1} (x_k)^{-1}}=t^N\Big(1+O\big((R_4)^{-1}\big)\Big),
$$
and also
$$
H\big((qw_j)^{-1}; Y\big)=1+O\big((R_4)^{-1}\big).
$$
 Thus, integrating over $w_j$, $j=1,\dots,m$, in \eqref{eq_x24} and then sending $R_4\to\infty$ we get
\begin{multline}
\label{eq_x26}
C(m) \cdot \frac{ t^{m(N+r-m)}}{(2\pi \i)^{r-m} r!} \oint\limits_{|w_{m+1}|=R_1}\dots\oint\limits_{|w_{r}|=R_1} \det \left[ \frac{1}{
 w_k-tw_\ell}\right]_{k,\ell=m+1}^r \prod_{j={m+1}}^r   \HH\big(w_j;X\big)\, \HH\big((qw_j)^{-1}; Y\big) dw_j,
\end{multline}
where $C(m)$ is the constant computed via
$$
 C(m)=\frac{1}{(2\pi \i)^m} \oint_{|w_1|=R}\dots\oint_{|w_m|=R} \det \left[ \frac{1}{
 w_k-tw_\ell}\right]_{k,\ell=1}^m d w_1\cdots dw_m
$$
(note that the exact value of $R>0$ is irrelevant in the last integral).

Further, for $j=m+1,\dots,r$ we expand the functions $\HH(w_j;X)$  into series using
$$
 \HH(w_j;X)=\exp\left(\sum_{k=1}^\infty (1-t^k) \frac{(w_j)^k p_k(X)}{k}\right)
$$
and the power series expansion of the exponential; similarly expand $H\big((qw_j)^{-1}; Y\big)$. We get
\begin{equation}
\label{eq_x27}
 \prod_{j={m+1}}^r   \HH\big(w_j;X\big)\, \HH\big((qw_j)^{-1}; Y\big)=\sum_{n=0}^{\infty} f_n(w_{m+1},\dots,w_{r})
 g_n,
\end{equation}
where $f_n$, $n\ge 0$, is an analytic function on the torus $w_j=R_1$, $j=m+1,\dots,r$, and $g_n$,
$n\ge 0$, is a homogeneous symmetric polynomial in $x_1,\dots,x_N$ and $y_1,\dots,y_N$ of degree
$n$, more precisely, $g_n$ is a polynomial in $p_k(X)$, $p_k(Y)$, whose coefficients do not depend
on $N$ or any choices we made.
 Note that
the convergence of expansions of $\HH(w_j;X)$, $\HH((qw_j)^{-1};Y)$ is uniform with respect to varying the $\{w_j\}$ on their contours,
the $\{x_j\}$ in the annulus $R_2<|x_j|^{-1}<R_3$ and the $\{y_j\}$ in some neighborhood of zero. Therefore, the order
of integration in \eqref{eq_x26} and summation in \eqref{eq_x27} can be interchanged. Hence,
evaluating the integrals over $w_{m+1},\dots,w_r$ transforms \eqref{eq_x26} into the sum
$$
 t^{m(N+r-m)} \sum_{n=0}^{\infty} \hat f_n
 g_n,
$$
where $g_n$ are as above, while $\hat f_n$ are certain coefficients which do not depend on $N$ and
are given by
\begin{equation*}
 \hat f_n= \frac{C(m)}{(2\pi \i)^{r-m} r! } \oint\limits_{|w_{m+1}|=R_1}\dots\oint\limits_{|w_{r}|=R_1} \det \left[ \frac{1}{
 w_k-tw_\ell}\right]_{k,\ell=m+1}^r f_n(w_{m+1},\dots,w_r)\prod_{j={m+1}}^r  dw_j.
\end{equation*}
If now $m\ge 1$, then the coefficients $t^{m(N+r-m)}\hat f_n$ vanish as $N\to\infty$. On the other
hand, for $m=0$ we arrive at the right side of \eqref{eq_main_lemma}.
\end{proof}

\subsection{Multilevel observable} The combination of Proposition
\ref{proposition_restated_computation_from_BigMac} and Proposition
\ref{proposition_evaluation_of_scalar_product} gives a way to compute the expectations of very
general observables of formal Macdonald processes. For two sets of variables $U=(u_1,u_2,\dots)$
and $V=(v_1,v_2,\dots)$ set
\begin{equation}\label{Wexp}
 \W(U;V) =\prod_{i,j} \frac{(1-t u_i v_j) (1-q u_i v_j)}{(1-u_i v_j)(1-qt u_i v_j)} = \exp\left(\sum_{k=1}^{\infty} \frac{(1-t^k)(1-q^k)}{k} p_k\left(U\right) p_k\left(
V\right) \right)
\end{equation}

\begin{theorem}
\label{theorem_observable_multilevel}
 Take $N\ge 1$ and $r_1,\dots,r_N\ge 0$. For $1\le m \le N$, set $V^m=\{v^m_1,\dots,v^m_{r_m}\}$ and define
$$
 DV^m = \frac{1}{(r_m)!(2\pi \i)^{r_m}} \det \left[ \frac{1}{
 v^m_i-tv^m_j}\right]_{i,j=1}^{r_m} \prod_{i=1}^{r_m} dv_i^m.
$$
We have
\begin{multline}
\label{eq_x6}
 \sum_{\lambda^1,\dots,\lambda^N\in\Y}
 \O_{r_1}(\lambda^1) \cdots
 \O_{r_N}(\lambda^N) \MPF_{N,\mathbf{A},\mathbf{B}} (\lambda^1,\dots,\lambda^N)
 \\= \oint \cdots \oint \prod_{\alpha=1}^N
 (DV^\alpha) \prod_{1\le \alpha\le \beta\le N} \HH\big((qV^\alpha)^{-1}; B^\beta\big)
  \HH\big(A^{\alpha};V^\beta\big)
  \W\big((qV^\alpha)^{-1}; V^\beta\big),
\end{multline}
where $\O_r(\lambda)$ is as in Definition \ref{Definition_observable}. Note that for a set of variables
$V=\{v_1,\dots,v_r\}$,  $(qV)^{-1}$ means the set $\{(qv_1)^{-1},\dots,(q v_r)^{-1}\}$. The
contours of integration are  a collection of positively oriented circles $\gamma_1,\dots,\gamma_m$ of
radii $R_1,\dots,R_m$ around the origin such that $v_i^\alpha$ is integrated over $\gamma_\alpha$,
and the radii are such that $R_\beta<qR_\alpha$ for $1\le \alpha<\beta \le N$.
\end{theorem}
Similarly to Proposition \ref{proposition_restated_computation_from_BigMac}, \eqref{eq_x6} should
be understood as an identity of elements of $
 \overline{\Lambda_{A^1}}\otimes \dots \otimes  \overline{\Lambda_{A^N}}\otimes
 \overline{\Lambda_{B^1}}\otimes \dots  \otimes \overline{\Lambda_{B^N}}
$. Such an element in the right side of \eqref{eq_x6} is obtained by expanding all
$\HH\big((qV^\alpha)^{-1}; B^\beta\big)$ and $\HH\big(A^{\alpha};V^\beta\big)$ into symmetric
series and then evaluating the integrals term--wise. This evaluation can be either done
analytically (i.e., computing complex contour--integrals) or algebraically by expanding the integrals
in series in variables $v_i^m$ and $(v_i^m)^{-1}$ using:
$$
\frac{1}{v^m_i-t v_j^m} = \frac{1}{v^m_i} \cdot \frac{1}{1-t v^m_i/v^m_j} =\frac{1}{v^m_i}
\sum_{k=0}^{\infty} \frac{
 t^k (v^m_j)^k}{(v^m_i)^{k}},
$$
and for $\alpha<\beta$
\begin{eqnarray*}
 \W\big((qv_i^\alpha)^{-1};v_j^\beta\big) &=&\frac{(1-t v_j^\beta/(qv_i^\beta)) (1-q v_j^\beta/(qv_i^\alpha))}{(1-v_j^\beta/(qv_i^\alpha))(1-qt v_j^\beta/(qv_i^\alpha))}
 \\
 &=& \left(1-t \frac{v_j^\beta}{qv_i^\alpha}\right) \left(1- \frac{v_j^\beta}{v_i^\alpha}\right) \left(\sum_{k=0}^{\infty}
 \left(\frac{v_j^\beta}{qv_i^\alpha}\right)^k\right)\cdot \left( \sum_{k=0}^{\infty}
 \left(t \frac{v_j^\beta}{v_i^\alpha}\right)^k\right),
\end{eqnarray*}
and then evaluating the coefficient of $\left(\prod_{m=1}^N \prod_{i=1}^{r_m} v_i^m\right)^{-1}$.

\begin{proof}[Proof of Theorem \ref{theorem_observable_multilevel}]
Using Proposition \ref{Prop_Mac_pro_alternative}, write the left--hand side of \eqref{eq_x6} as
\begin{multline*}
\frac{1}{ \prod_{1\le \alpha\le \beta \le N} \Pi(A^\alpha; B^\beta)}\, \cdot
 \l \sum_{\lambda^{N}\in\Y} \O_{r_N}(\lambda^N) Q_{\lambda^N}(B^N) P_{\lambda^N}(A^N,Y^{N-1}),\\
 \l \sum_{\lambda^{N-1}\in\Y} \O_{r_{N-1}}(\lambda^{N-1}) Q_{\lambda^{N-1}}(Y^{N-1},B^{N-1})
   P_{\lambda^{N-1}}(A^{N-1},Y^{N-2}),\\ \vdots\\
 \l \sum_{\lambda^2\in\Y} \O_{r_2}(\lambda^2) Q_{\lambda^2}(Y^2,B^2) P_{\lambda^2}(A^2,Y^1),
 \sum_{\lambda^1\in\Y} \O_{r_1}(\lambda^1) Q_{\lambda^1}(Y^1,B^1) P_{\lambda^1}(A^1) \r_{Y^1}
 \dots  \r_{Y^{N-2}} \r_{Y^{N-1}}  .
\end{multline*}
Applying Proposition \ref{proposition_restated_computation_from_BigMac} one time for each of the
summations over $\lambda^1, \dots, \lambda^N$, we find that the above expression equals
\begin{multline}
\label{eq_x19}
\frac{1}{ \prod_{1\le i\le j \le N} \Pi(A^i; B^j)}\, \cdot
 \l \oint DV^N \HH\big(V^N; A^N,Y^{N-1}\big) \HH\big((qV^N)^{-1};B^N\big) \Pi\big(B^N; A^N,Y^{N-1}\big),\\
 \l \oint DV^{N-1} \HH\big(V^{N-1}; A^{N-1},Y^{N-2}\big) \HH\big((qV^{N-1})^{-1};B^{N-1},Y^{N-1}\big) \Pi\big(B^{N-1},Y^{N-1}; A^{N-1},Y^{N-2}\big)
 ,\\ \vdots\\
 \l \oint DV^{2} \HH\big(V^{2}; A^{2},Y^{1}\big) \HH\big((qV^{2})^{-1};B^{2},Y^{2}\big) \Pi\big(B^{2},Y^{2}; A^{2},Y^{1}\big)
 ,\\ \oint DV^{1} \HH\big(V^{1}; A^{1}\big) \HH\big((qV^{1})^{-1};B^{1},Y^{1}\big) \Pi\big(B^{1},Y^{1}; A^{1}\big) \r_{Y^1} \dots  \r_{Y^{N-2}} \r_{Y^{N-1}}.
\end{multline}
Note that if we view the integrations as algebraic operations (as is explained after Proposition
\ref{proposition_restated_computation_from_BigMac}), then in \eqref{eq_x19} using the continuity of
the Macdonald pairing and of the constant term evaluation in the topology of completed graded
algebras, we can interchange the order of integration and evaluating scalar products. Then we can
use Proposition \ref{proposition_evaluation_of_scalar_product}. For the variables $Y^1$ we get
(omitting all the factors independent of $Y^1$ which do not change in the scalar product
evaluation)
\begin{multline*}
 \l \HH\big(V^{2}; Y^{1}\big) \Pi\big(B^{2},Y^{2}; Y^{1}\big),
  \HH\big((qV^{1})^{-1};Y^{1}\big) \Pi\big(Y^{1}; A^{1}\big)\r_{Y^1}
\\= \HH\big(V^2;A^1\big) \HH\big((qV^1)^{-1};B^2,Y^2\big) \Pi\big(B^2,Y^2; A^1\big)\, \exp\left(\sum_{k=1}^{\infty}
\frac{(1-t^k)(1-q^k)}{k} p_k\left(V^2\right) p_k\left( (qV^1)^{-1}\right) \right).
\end{multline*}
Note that by \eqref{Wexp}
$$
\exp\left(\sum_{k=1}^{\infty} \frac{(1-t^k)(1-q^k)}{k} p_k\left(V^2\right) p_k\left(
(qV^1)^{-1}\right) \right)=W\left(V^2;(qV^1)^{-1}\right),
$$
if we assume $|v^2_i/(qv^1_j)|<1$ when expanding $W(V^2;(qV^1)^{-1})$ in power series. This gives
the same restriction on the contours as the one in Theorem \ref{theorem_observable_multilevel}. In
the next step we evaluate the scalar product for the variables $Y^2$ and find (again omitting
factors independent of $Y^2$)
\begin{multline*}
 \l \HH\big(V^{3}; Y^{2}\big) \Pi\big(B^{3},Y^{3}; Y^{2}\big),
  \HH\big((qV^{2})^{-1}, (qV^1)^{-1} ;Y^{2}\big) \Pi\big(Y^2; A^1, A^{2}\big) \r_{Y^2}
\\= \HH\big(V^3; A^1,A^{2}\big)\, \HH\big( (qV^{2})^{-1}, (qV^1)^{-1}; B^{3},Y^{3}\big)\, \Pi\big( B^{3},Y^{3}; A^1, A^{2}\big)
\, \W\big(V^{3}; (qV^{2})^{-1}, (qV^1)^{-1}\big).
\end{multline*}
Further evaluating scalar products for variables $Y^3$, \dots, $Y^{N-1}$ we arrive at the claimed
formula.
\end{proof}

\subsection{Simple corollaries}
Let us give two corollaries of Theorem \ref{theorem_observable_multilevel}.

\begin{corollary}
\label{corollary_observable_multilevel_multiplied}
 Take any $M$ integers $1\le k_1\le k_2 \le \dots
\le k_M\le N$ and $M$ positive integers $r_1,\dots,r_M$. With the notations and contours as in
Theorem \ref{theorem_observable_multilevel} we have
\begin{multline}
\label{eq_x23}
 \sum_{\lambda^1,\dots,\lambda^N}  \O_{r_1}(\lambda^{k_1}) \cdots
 \O_{r_M}(\lambda^{k_M}) \MPF_{N,\mathbf{A},\mathbf{B}} (\lambda^1,\dots,\lambda^N)
 = \oint \cdots \oint \prod_{m=1}^M
 (DV^m) \\ \times \prod_{1\le \alpha,\beta\le M:\,k_\alpha\le \beta} \HH\big((qV^\alpha)^{-1}; B^\beta\big)
 \prod_{1\le \alpha,\beta \le M:\, \alpha\le k_\beta}  \HH\big(A^{\alpha};V^\beta\big)
   \prod_{1\le \alpha<\beta\le M} \W\big((qV^\alpha)^{-1}; V^\beta\big).
\end{multline}
\end{corollary}
{\bf Remark. } The difference from Theorem \ref{theorem_observable_multilevel} is that now we
compute expectations of various products and powers of $\O_r(\lambda^m)$, thus \eqref{eq_x23} is
more general than \eqref{eq_x6}.
\begin{proof}[Proof of Corollary \ref{corollary_observable_multilevel_multiplied}]
The proof is a combination of Theorem \ref{theorem_observable_multilevel} with Proposition
\ref{prop_Mac_pro_empty}.

Take $2(N+M)$ auxiliary sets of variables $\mathbf{C}=(C_1,\dots,C_{N+M})$,
$\mathbf{D}=(D_1,\dots,D_{N+M})$.

Let $\lambda^1,\dots,\lambda^{N+M}$ be distributed according to $\MPF_{M,\mathbf{C},\mathbf{D}}$
and apply
 Theorem \ref{theorem_observable_multilevel} to it
 with the sequence of numbers
 $r'_i$, $i=1,\dots,N+M$ (they were called $r_i$ in Theorem \ref{theorem_observable_multilevel}, but we use $r'_i$ here to avoid the confusion
 with numbers $r_i$ of Corollary
  \ref{corollary_observable_multilevel_multiplied}) obtained as follows:
  we set the first $k_1$ $r'_i$'s to equal
$0$,
  the next one (i.e., $r'_{k_1+1}$)) is $r_1$,
  then we take $k_2-k_1$ zeroes, then $r_2$, \dots, so on until $r_M$ and finally $N-k_M$ zeroes. Applying to the result
  $\phi^{i-1}_0$ (as in Proposition \ref{prop_Mac_pro_empty})
  for all indices $1\le i\le N+M$ such that
  $r_i\ne 0$, and renaming the remaining sets of variables $C_j$, $D_j$ into $A_i$ and $B_i$,
 we get \eqref{eq_x23}.

 For example, if $N=1$, $M=2$, and $k_1=k_2=1$, $r_1=r_2=1$, then we start from $\mathbf{C}=(C_1,C_2,C_{3})$,
$\mathbf{D}=(D_1,D_2,D_{3})$ and the corresponding Macdonald process. Application of Theorem
\ref{theorem_observable_multilevel} with $r'=(0,1,1)$ gives the contour--integral formula for
\begin{equation}
\label{eq_x33}
 \sum_{\lambda^1,\lambda^2,\lambda^3\in Y} \O_{1}(\lambda^2) \O_1(\lambda^3) \MPF_{3;\mathbf C, \mathbf
 D}(\lambda^1,\lambda^2,\lambda^3).
\end{equation}
When we apply $\phi^{2}_0$ and $\phi^{1}_0$ to \eqref{eq_x33}, the summation becomes restricted to
$\lambda^1=\lambda^2=\lambda^3$ and after renaming the sets of variables we arrive at the desired
contour--integral formula for
$$
 \sum_{\lambda^1\in \Y} \left(\O_{1}(\lambda^1)\right)^2 \MPF_{1;\mathbf A, \mathbf
 B}(\lambda^1). \qedhere
$$
\end{proof}

\begin{corollary}
\label{Corollary_multiplied_by_numbers} In the notations of Theorem
\ref{theorem_observable_multilevel}, let $c_1,\dots, c_N$ be any numbers (or formal variables) and
set
$$
 d_i=\prod_{j=i}^N c_i,\quad d_{N+1}=1.
$$
We have
\begin{multline}
\label{eq_x8}
 \sum_{\lambda^1,\dots,\lambda^N}
  \Bigl((c_1)^{|\lambda^1|}\O_{r_1}(\lambda^1)\Bigr) \cdots
 \Bigl((c_N)^{|\lambda^N|}\O_{r_N}(\lambda^N)\Bigr)
 \MPF_{N,\mathbf{A},\mathbf{B}}(\lambda^1,\dots,\lambda^N)
 \\ = \prod_{1\le\alpha \le \beta\le N}
 \frac{ \Pi(d_\alpha A^\alpha; (d_{\beta+1})^{-1}B^\beta)} {\Pi(A^\alpha; B^\beta)} \oint \cdots \oint
\prod_{m=1}^N
 (DV^m)\\ \times \prod_{1\le\alpha \le \beta\le N} \HH\big((qV^\alpha)^{-1};(d_{\beta+1})^{-1} B^\beta\big)\,\HH\big(d_\alpha A^{\alpha};V^\beta\big)\, \W\big((qV^\alpha)^{-1}; V^\beta\big).
\end{multline}
\end{corollary}
\begin{proof}
The homogeneity of (skew) Macdonald symmetric functions implies that
\begin{multline*}
\frac{d_1^{|\lambda^1|} d_2^{|\lambda^2|-|\lambda^1|} \cdots d_N^{|\lambda^N|-|\lambda^{N-1}|}}{
\prod_{\alpha\le \beta} \Pi(A^\alpha; B^\beta)} P_{\lambda^1}(A^1)
\Psi_{\lambda^2,\lambda^{1}}(A^2;B^{1}) \cdots \Psi_{\lambda^N,\lambda^{N-1}}(A^N;B^{N-1})
Q_{\lambda^N}(B^N)
\\
= \frac{1}{\prod_{\alpha\le \beta} \Pi(A^\alpha; B^\beta)} P_{\lambda^1}(d_1A^1)
\Psi_{\lambda^2,\lambda^{1}}(d_2 A^2;(d_2)^{-1}B^{1}) \cdots \Psi_{\lambda^N,\lambda^{N-1}}(d_N
A^N; (d_N)^{-1}B^{N-1}) Q_{\lambda^N}(B^N).
\end{multline*}
Thus, we can use Proposition \ref{Prop_Mac_pro_alternative} and \eqref{eq_x6} to compute the sum
in the left side of \eqref{eq_x8} and we reach the desired result.
\end{proof}

Of course, one can also combine Corollaries \eqref{corollary_observable_multilevel_multiplied} and
\eqref{Corollary_multiplied_by_numbers}. We leave the resulting statement to an interested reader.

\section{Ascending Macdonald processes}\label{ascsec}
 Let us focus on a special case of Macdonald processes, that is very useful in applications, cf.\ \cite{BigMac}.

 For any complex  number $a$, let $\phi_a$ be an algebra homomorphism $\phi_a:\Lambda\to
 \mathbb C$ (i.e., a specialization), such that $\phi_a(p_k)=a^k$. In other words, $\phi_a$ is the
 substitution $x_1=a$, $x_2=x_3=\dots=0$ into a symmetric function $f(X)$ in variables
 $X=(x_1,x_2,\dots)$. In particular, $\phi_0$ is the evaluation of the free term of a symmetric
 polynomial, as before. Let us also fix an arbitrary specialization $\rho:\Lambda\to \mathbb C$. As
 explained in Section \ref{section_formal_measures}, $\phi_a$ and $\rho$ can be naturally extended so as to act on elements of $\Lambda_{\phi_a}\subset \overline{\Lambda}$ and $\Lambda_\rho\subset \overline{\Lambda}$, respectively. In
 what follows (and where it leads to no confusion) we write $f(a)$ and $f(\rho)$ for $\phi_a(f)$
 and $\rho(f)$, respectively.

 Recall that values
 $\MPF_{N,\mathbf{A},\mathbf{B}}(\lambda^1,\dots,\lambda^N)$
 of a formal Macdonald process belong to the completed tensor product
 $$\overline{\Lambda_{A^1}}\otimes \dots \otimes  \overline{\Lambda_{A^N}}\otimes
 \overline{\Lambda_{B^1}}\otimes \dots  \otimes \overline{\Lambda_{B^N}}.
 $$

\begin{definition}
\label{def_Mac_ascending_proc}
 Take $N$ non--zero complex numbers $a_1$,\dots, $a_N$ and a specialization $\rho$ such that for some $0<R<1$
 we have:
 $$
  |p_k(\rho)|<R^k,\ k=1,2,\dots, \quad \quad |a_i|R<1,\, i=1,\dots,N.
 $$
The ascending Macdonald process $\MP_{N;\{a_i\};\rho}$ is defined as a composition of the formal
Macdonald
 process
 $\MPF_{N,\mathbf{A},\mathbf{B}}$ and the map $\Theta$:
 $$
  \Theta=\phi_{a_1}\otimes\phi_{a_2}\otimes\dots\otimes \phi_{a_N}\otimes \phi_0
  \otimes\dots\otimes \phi_0 \otimes \rho.
 $$
 In other words,  $\MM_{N;\{a_i\};\rho}$ is a complex--valued measure on $\Y^N$ which sums to $1$ and such that
 $$
  \MP_{N;\{a_i\};\rho}(\lambda^1,\dots,\lambda^N)=\dfrac{
  P_{\lambda_1}(a_1) P_{\lambda^2/\lambda^1}(a_2)\cdots P_{\lambda^N/\lambda^{N-1}}(a_N) Q_{\lambda^N}(\rho) }{\Pi(a_1,\dots,a_N;\rho)},
 $$
 where
 \begin{equation}
 \label{eq_Pi_specialized}
  \Pi(a_1,\dots,a_N;\rho)=\prod_{i=1}^N \Pi(a_i;\rho) = \prod_{i=1}^N \exp\left(\sum_{k=1}^\infty
  \frac{1-t^k}{1-q^k} \cdot \frac{ \sum_{i=1}^N (a_i)^k p_k(\rho)}{k} \right).
 \end{equation}
\end{definition}
{\bf Remark 1.} Our restrictions on $\rho$ and $a_i$ ensure the absolute convergence of the series
$$
 \sum_{\lambda^1,\dots,\lambda^N\in\Y} P_{\lambda_1}(a_1) P_{\lambda^2/\lambda^1}(a_2)\cdots P_{\lambda^N/\lambda^{N-1}}(a_N)
 Q_{\lambda^N}(\rho).
$$
Indeed, this series is a permutation of the absolutely convergent series obtained by expanding the
right--hand side of \eqref{eq_Pi_specialized} in power series in $a_i$, $p_k$, thus it is also
absolutely convergent.

{\bf Remark 2.} If $a_i$ are non--negative reals and $\rho$ takes nonnegative values on Macdonald
symmetric functions, then $\MP_{N;\{a_i\};\rho}$ is a probability measure, i.e., it is positive,
cf.\ \cite[Definition 2.2.7]{BigMac}.

{\bf Remark 3.} $\MP_{N;\{a_i\};\rho}(\lambda^1,\dots,\lambda^N)$ is an analytic function of
the complex numbers $a_i$, $i=1,\dots,N$.

\begin{lemma}
 The support of the ascending Macdonald process $\MP_{N;\{a_i\};\rho}$ is (a subset of) the set of
 sequences $\lambda^1,\dots,\lambda^N$ such that $\ell(\lambda^i)\le i$ for
 $i=1,\dots,N$ and
 $$
  \lambda^i_1\ge \lambda^{i-1}_2 \ge \lambda^i_2 \ge \dots\ge \lambda^{i-1}_{i-1} \ge
  \lambda^{i}_i,
 $$
 for $i=2,\dots,N$ (thus the term ``ascending'').
\end{lemma}
\begin{proof} This follows from the fact that for any $a\in\mathbb C\setminus\{0\}$, $P_{\lambda/\mu}(a)=0$
unless
 $$
  \lambda_1\ge\mu_1\ge\lambda_2\ge\mu_2\ge\lambda_3\ge\dots,
 $$
which can be found in \cite[Section 7, Chapter VI]{M}.
\end{proof}

Note that the projection of the ascending Macdonald process $\MP_{N;\{a_i\};\rho}$ to $\lambda^N$
is a \emph{Macdonald measure}, cf.\ Proposition \ref{prop_Mac_pro_restriction}, \cite[Proposition
6.3]{BG}, \cite[Section 2.2.2]{BigMac}, that we now define.

\begin{definition}
 Take $N$ non--zero complex numbers $a_1$,\dots, $a_N$ and a specialization $\rho$ such that for some $0<R<1$
 we have:
 $$
  |p_k(\rho)|<R^k,\ k=1,2,\dots, \quad \quad |a_i|R<1,\, i=1,\dots,N.
 $$
The Macdonald measure $\MM_{N;\{a_i\};\rho}$ is a complex--valued measure on $\Y$ which sums to $1$
such that
 $$
  \MM_{N;\{a_i\};\rho}(\lambda)=\dfrac{
  P_{\lambda}(a_1,\dots,a_N) Q_{\lambda}(\rho)}{\Pi(a_1,\dots,a_N;\rho)}.
 $$
\end{definition}

The projection of the ascending Macdonald process $\MP_{N;\{a_i\};\rho}$ to
$\lambda^1,\dots,\lambda^k$ is again an ascending Macdonald process $\MP_{k;\{a_i\};\rho}$, cf.\
Proposition \ref{prop_Mac_pro_restriction} and \cite[Section 2.2.2]{BigMac}.

\medskip

The rest of this section is devoted to computing expectations of observables of ascending Macdonald
processes (and measures). First, we present analogues of Theorem
\ref{theorem_observable_multilevel} and Corollary \ref{corollary_observable_multilevel_multiplied},
which in this case can be proved in a much simpler way that does not require the formal approach.
Then we show how the expectation of another (much smaller) family of observables can be written in
terms of \emph{Fredholm determinants}.

\subsection{Multilevel moments}

\label{Section_ascending_obs}

Let us introduce certain difference operators which act on analytic functions in $x_1,\dots,x_N$
invariant (symmetric) under the permutations of arguments.

For any subset $I\subset\{1,\dots,N\}$ define
$$
 A_I(x_1,\dots,x_N;t) = \prod_{i\in I} \prod_{j\not \in I} \frac{x_i-tx_j}{x_i-x_j}.
$$
Define the shift operator $T_{q,i}$ through
$$
 [T_{q,i} f](x_1,\dots,x_N) = f(x_1,\dots,x_{i-1}, qx_i,x_{i+1},\dots,x_N).
$$
For any $1\le r\le N$ define the $r$th \emph{Macdonald difference operator} $\Mac^r_N$ through
$$
\Mac^r_N = \sum_{I\subset\{1,\dots,N\}:\, |I|=k} A_I(x_1,\dots,x_N;t) \prod_{i\in I} T_{q,i}.
$$
 One of the important properties of the Macdonald difference operators is the fact that the Macdonald
 polynomials are their eigenfunctions, see \cite[Section 4, Chapter VI]{M}:
 \begin{equation}
 \label{eq_Macdonald_operator_eigen}
  \Mac^r_kP_\lambda(x_1,\dots,x_k;q,t) = e_r(q^{\lambda_1}t^{k-1},q^{\lambda_2}t^{k-2},\dots,q^{\lambda_k})
  P_\lambda(x_1,\dots,x_k;q,t).
 \end{equation}

The following proposition is a key in evaluating expectations of observables of ascending Macdonald
processes.

\begin{proposition}
\label{prop_observable_as_application_of_op}
 Fix sequence of integers $N\ge n_1 \ge n_2\ge\dots\ge n_m\ge 1$ and $r_1,\dots,r_m$, such
 that $0\le r_i \le n_i$ for $i=1,\dots,m$. Then
 \begin{multline}
  \label{eq_observable_as_application_of_op}
  \sum_{\lambda^1,\dots,\lambda^N\in\Y} \prod_{i=1}^m
  e_{r_i}(q^{\lambda^{n_i}_1}t^{n_i-1},q^{\lambda^{n_i}_2}t^{n_i-2},\dots,q^{\lambda^{n_i}_{n_i}}) \MP_{N;\{a_i\};\rho}(\lambda^1,\dots,\lambda^N) \\=
   \dfrac {\Mac^{r_m}_{n_m} \cdots \Mac^{r_1}_{n_1}
   \Pi(x_1,\dots,x_N;\rho)}{\Pi(x_1,\dots,x_N;\rho)}\Bigg|_{x_1=a_1,\dots,x_N=a_N}.
 \end{multline}
\end{proposition}
{\bf Remark.} The operators $\Mac^r_n$ do \emph{not} commute for different $n$s. Thus the order of
operators in the right side of \eqref{eq_observable_as_application_of_op} is important.
\begin{proof}[Proof of Proposition \ref{prop_observable_as_application_of_op}]
We may immediately replace $N$ on the right-hand side of \eqref{eq_observable_as_application_of_op} with $n_1$ since $\Pi$ is multiplicative in the $x$ variables, and the difference operators only act on $x_1,\ldots, x_{n_1}$.
 We expand $\Pi(x_1,\dots,x_{n_1};\rho)$ in a sum using a specialized version of \eqref{eq_Cauchy_for_Mac}:
 \begin{equation}
 \label{eq_x10}
  \Pi(x_1,\dots,x_{n_1};\rho)=\sum_{\lambda\in\Y:\ell(\lambda)\le n_1} P_\lambda(x_1,\dots,x_{n_1}) Q_\lambda(\rho).
 \end{equation}
 Apply
 $\prod_{i=1}^{k-1} \Mac^{r_i}_{n_i}$
  to the sum, where $k$ is a maximal number such that $n_1=n_{2}=\dots=n_{k-1}$.  Using \eqref{eq_Macdonald_operator_eigen} we get
 \begin{multline}
 \label{eq_x20}
  \prod_{i=1}^{k-1} \Mac^{r_i}_{n_i} \Pi(x_1,\dots,x_{n_1};\rho)\\=
  \sum_{\lambda^{n_1}\in\Y:\ell(\lambda)\le n_1} \left(\prod_{i=1}^{k-1}
  e_{r_i}\left(q^{\lambda^{n_1}_1}t^{{n_1}-1},\dots,q^{\lambda^{n_1}_{n_1}}\right)\right)
  P_{\lambda^{n_1}}(x_1,\dots,x_{n_1}) Q_{\lambda^{n_1}}(\rho).
 \end{multline}
 Note that since $0<q,t<1$,
 $$\left|e_{r}\left(q^{\lambda^{n_1}_{1}}t^{{n_1}-1},\dots,q^{\lambda_{n_1}}\right)\right|\le (n_1)^r.$$
 Therefore, the series in \eqref{eq_x20} is absolutely convergent. Now substitute in \eqref{eq_x20} the decomposition (which is a specialized version of the
 definition \eqref{eq_skew})
 $$
 P_{\lambda^{n_1}}(x_1,\dots,x_{n_1})= \sum_{\lambda^{n_k}\in\Y: \ell(\lambda^{n_k})\le n_k}
 P_{\lambda^{n_k}}(x_1,\dots,x_{n_k}) P_{\lambda^{n_1}/\lambda^{n_k}}(x_{n_k+1},\dots,x_{n_1})
 $$
 and apply (again using \eqref{eq_Macdonald_operator_eigen})
 $\prod_{i=k}^{h-1} \Mac^{r_i}_{n_i}$
 to the resulting sum, where $h$ is a maximal number such that $n_k=n_{k+1}=\dots=n_{h-1}$.
 Iterating this procedure we arrive at the
 desired statement.
\end{proof}

The next two theorems express averages of a class of observables of  ascending Macdonald processes
through contour--integrals.

\begin{theorem}
\label{theorem_ascending_multilevel_1}
 Take $N\ge 1$, non--zero complex numbers $a_1$,\dots, $a_N$, and a specialization $\rho$ such that for some $0<R<1$
 we have:
 $$
  |p_k(\rho)|<R^k,\ k=1,2,\dots, \quad \quad |a_i|R<1,\, i=1,\dots,N.
 $$
 Fix sequences of integers $N\ge n_1 \ge n_2\ge\dots\ge n_m\ge 1$ and $r_1,\dots,r_m$, such
 that $0\le r_i \le n_i$ for $i=1,\dots,m$. Suppose that there exist closed positively oriented contours
 $\gamma^{\alpha}_i$, $\alpha=1,\dots,m$, $i=1,\dots,r_\alpha$, in the complex plane such that
 \begin{enumerate}[(I)]
 \item All contours lie inside $\mathbb D_{R^{-1}}=\{z\in\mathbb C\mid |z|<R^{-1}\}$.
\item All contours enclose the  points $a_1,\dots,a_N$, but not $0$.
\item The contour $\gamma^\alpha_i$ does not enclose  points $t^{-1} q^{k}a_1,\dots,
t^{-1}q^{k}a_N$, $k=0,\dots,m-\alpha$.
\item  The contour $\gamma^{\alpha}_i$ intersects neither the interior of the image of $\gamma^{\beta}_j$ multiplied by $q$,
nor the interior of the image of $\gamma^{\beta}_j$ multiplied by $t^{-1}$
 for $i=1,\dots,r_\alpha$, $j=1,\dots,r_\beta$ and $\beta>\alpha$.
 \end{enumerate}
Then we have
\begin{align}
\label{eq_x28}
\nonumber&  \sum_{\lambda^1,\dots,\lambda^N} \prod_{i=1}^m
  e_{r_i}(q^{\lambda^{n_i}_1}t^{n_i-1},\dots,q^{\lambda^{n_i}_{n_i}})
  \MP_{N;\{a_i\};\rho}(\lambda^1,\dots,\lambda^N)\\
  &= \prod_{\alpha=1}^m \frac{1}{(2\pi \i)^{r_\alpha} r_\alpha!} \oint\dots\oint
  \prod_{1\le \alpha <\beta \le m} \left(\prod_{i=1}^{r_\alpha} \prod_{i=1}^{r_\beta}
  \frac{(tz^{\alpha}_i-q z^{\beta}_j)(z^{\alpha}_i-z^{\beta}_j)}{(z^\alpha_i-q z^{\beta}_j)(t
  z^{\alpha}_i - z^{\beta}_j)}\right)\\
  \nonumber& \times \prod_{\alpha=1}^m \left(
  \det\left[\frac{1}{tz^{\alpha}_i-z^{\beta}_j}\right]_{i,j=1}^{r_\alpha} \prod_{i=1}^{r_\alpha}
  \left(\left(\prod_{j=1}^{n_\alpha} \frac{tz^{\alpha}_i-a_j}{z^{\alpha}_i-a_j}\right)
  \frac{\Pi(qz^{\alpha}_i;\rho)}{\Pi(z^{\alpha}_i;\rho)}dz^{\alpha}_i\right)\right)
 \end{align}
 where $z^{\alpha}_i$ is integrated over $\gamma^{\alpha}_i$.
\end{theorem}
{\bf Remark 1.} When all $a_i$ are equal, $t\ll 1$ and $q$ is close to $1$, then the required
contours do exist, see e.g.\ \cite[Figure 3.1]{BigMac}. However, for some specific choices of
$a_i$, $t$ and $q$ the desired contours may fail to exist, see \cite{BG_GFF} for the study of one
such case.

{\bf Remark 2.} In the case $n_1=n_2=\dots=n_m$ this theorem is equivalent to \cite[Proposition
2.2.15]{BigMac}.

\begin{proof}[Proof of Theorem \ref{theorem_ascending_multilevel_1}] The proof is similar to that of \cite[Proposition
2.2.15]{BigMac}.

 We use Proposition \ref{prop_observable_as_application_of_op} and sequentially
apply the operators $\Mac^{r_i}_{n_i}$. At the first step we need the following identity:
\begin{multline}
 \Mac^{r_1}_{n_1}
   \Pi(x_1,\dots,x_N;\rho)= \\
   \frac{\Pi(x_1,\dots,x_N;\rho)}{(2\pi \i) (r_1)!} \oint\dots\oint
  \det\left[\frac{1}{tz^1_i-z^1_j}\right]_{i,j=1}^{r} \prod_{i=1}^{r_1}
  \left(\left(\prod_{j=1}^{n_1} \frac{tz^1_i-x_j}{z^1_{i}-x_j}\right)
  \frac{\Pi(qz^1_i;\rho)}{\Pi(z^1_i;\rho)}dz^1_{i}\right)
  \label{eq_x21}
\end{multline}
with $z^1_i$ integrated over $\gamma^1_i$ and with $x_1,\dots,x_N$ being arbitrary complex numbers
inside $\gamma^1_i$ such that $t^{-1} x_i$, $i=1,\dots,N$, are outside $\gamma^1_i$ (note that we
will later need $x_i$ to be equal to $q^k a_i$, for various $k\ge 0$, hence,  restriction $III$ on
the contours). The formula \eqref{eq_x21} is proved by expanding the right side of \eqref{eq_x21}
as a sum of residues, for that we need $\Pi(qz;\rho)/\Pi(z;\rho)$ to be analytic inside the
contour, hence restriction $I$ on the contours.

We claim that the only poles of the integrand inside the contour are at the points $z_i^1=x_j$,
$i=1,\dots,r$, $j=1,\dots,n$ (our choice of $\{x_j\}$ and restriction $II$ on the contours
guarantees that these points will always be inside, while $0$ is outside). To see this observe that
the only other possible singularities of the integrand arise when $z^1_{i_1}=x_j$,
$z^1_{i_2}=tx_j$, $z^1_{i_3}=t^2 x_j,\dots$ or $z^1_{i_1}=x_j$, $z_{i_2}=t^{-1}x_j$,
$z_{i_3}=t^{-2} x_j,\dots$ (here $i_1,i_2,\ldots$ are unique elements of $\{1,\ldots, r\}$ and
$j\in \{1,\ldots, n_1\}$). The first possibility actually does not produce a pole because of the
factor $(tz^1_i-x_j)$ in the integrand, and the second possibility is outside the contour because
of our choice of $\{x_j\}$. These considerations imply the claim.

The residues of the integrand at the points $z^1_i=x_j$, $i=1,\dots,r$, $j=1,\dots,n$ give the same
summands as in the definition of the operator $\Mac^{r_1}_{n_1}$ (note the factor $r_1!$ which
appears because we sum over subsets in the definition of $\Mac^{r_1}_{n_1}$ and over \emph{ordered}
subsets when expanding the integral as a sum of residues).

Next, we apply $\Mac^{r_2}_{n_2}$ to the right side of \eqref{eq_x21}. Note that by linearity we
can apply the difference operator under the integral. The part of the right--hand side of
\eqref{eq_x21} which is dependent on the $\{x_j\}_{j=1}^{n_2}$ is
$$
 \Pi(x_1,\dots,x_{n_2};\rho) \left(\prod_{i=1}^{r_1}\prod_{j=1}^{n_2}
 \frac{tz^1_i-x_j}{z^1_{i}-x_j}\right)=\prod_{j=1}^{n_2} \left( \Pi(x_j;\rho) \prod_{i=1}^{r_1}
 \frac{tz^1_i-x_j}{z^1_{i}-x_j}\right)=:\prod_{j=1}^{n_2}\Pi'(x_j,\rho).
$$
We can use an analogue of \eqref{eq_x21} to express
$$
\Mac^{r_2}_{n_2}\left(\prod_{j=1}^{n_2} \Pi'(x_j;\rho) \right)
$$
using the contours $\gamma^2_i$,
$i=1,\dots,r_2$. We have
$$
\frac{\Pi'(qz;\rho)}{\Pi'(z;\rho)}= \frac{\Pi(qz;\rho)}{\Pi(z;\rho)}\cdot
\frac{tz^1_i-qz}{z^1_{i}-qz} \cdot \frac{z^1_{i}-z} {tz^1_i-z},
$$
hence the beginning of the appearance of the product $\prod_{1\le \alpha <\beta \le m}$ in
\eqref{eq_x28}. Further, $\Pi'(qz;\rho)/\Pi'(z;\rho)$ must be analytic inside the contours
$\gamma^2_i$, hence, the restriction $IV$ on contours.

Further iterating this procedure for $\Mac^{r_3}_{n_3}$, \dots, $\Mac^{r_m}_{n_m}$ leads to integrals
over the contours $\gamma^3_i$, \dots $\gamma^m_i$ and ultimately \eqref{eq_x28}.
\end{proof}

The next statement is a version of Theorem \ref{theorem_ascending_multilevel_1} with a different
set of observables.

\begin{theorem}
\label{theorem_observable_with_vars_inversed}
 Take $N$ non--zero complex numbers $a_1$,\dots,
  $a_N$ and a specialization $\rho$ such that for some $0<R<1$ and $m\ge 1$
 we have:
 $$
  |p_k(\rho)|<R^k,\ k=1,2,\dots, \quad \quad |a_i|R<q^m,\, i=1,\dots,N.
 $$
 Fix two sequences of integers $N\ge n_1 \ge n_2\ge\dots\ge n_m\ge 1$ and $r_1,\dots,r_m$, such
 that $0\le r_i \le n_i$ for $i=1,\dots,m$. Suppose that there exist closed positively oriented contours
 $\gamma^{\alpha}_i$, $\alpha=1,\dots,m$, $i=1,\dots,r_\alpha$ such that
 \begin{enumerate}[(I)]
 \item All contours lie inside $\mathbb D_{qR^{-1}}=\{z\in\mathbb C\mid |z|<qR^{-1}\}$.
\item All contours enclose the points $a_1,\dots,a_N$, but not $0$.
\item The contour $\gamma^\alpha_i$ does not enclose  points $t q^{-k}a_1,\dots,
t q^{-k}a_N$, $k=0,\dots,m-\alpha$.
\item  The contour $\gamma^{\alpha}_i$ intersects neither the interior of the image of $\gamma^{\beta}_j$ multiplied by $q^{-1}$,
nor the interior of the image image of $\gamma^{\beta}_j$ multiplied by $t$
 for $i=1,\dots,r_\alpha$, $j=1,\dots,r_\beta$ and $\beta>\alpha$.
 \end{enumerate}
Then we have
\begin{align}
 \label{eq_observable_with_vars_inversed}
\nonumber&  \sum_{\lambda^1,\dots,\lambda^N\in\Y} \prod_{i=1}^m
  e_{r_i}(q^{-\lambda^{n_i}_1}t^{1-n_i},\dots,q^{-\lambda^{n_i}_{n_i}}) \MP_{N;\{a_i\};\rho}(\lambda^1,\dots,\lambda^N)
  \\&= \prod_{\alpha=1}^m \frac{1}{(2\pi \i)^{r_\alpha} r_\alpha!} \oint\dots\oint
  \prod_{1\le \alpha <\beta \le m} \left(\prod_{i=1}^{r_\alpha} \prod_{i=1}^{r_\beta}
  \frac{(z^{\alpha}_i-tq^{-1} z^{\beta}_j)(z^{\alpha}_i-z^{\beta}_{j})}{(z^{\alpha}_{i}-q^{-1}
  z^{\beta}_{j})(z^{\alpha}_{i} -t z^{\beta}_j)}\right)
  \\\nonumber& \times \prod_{\alpha=1}^m \left(
  \det\left[\frac{1}{t^{-1}z^{\alpha}_i-z^{\alpha}_j}\right]_{i,j=1}^{r_\alpha} \prod_{j=1}^{r_\alpha}
  \left(\left(\prod_{i=1}^{n_\alpha} \frac{t^{-1}z^{\alpha}_j-a_i}{z^{\alpha}_j-a_i}\right)
  \frac{\Pi(q^{-1}z^{\alpha}_j;\rho)}{\Pi(z^{\alpha}_j;\rho)}dz^\alpha_j\right)\right)
 \end{align}
 where $z^{\alpha}_i$ is integrated over $\gamma^{\alpha}_i$.
\end{theorem}
\begin{proof}
 First, note that since $0<q,t<1$,
 $$|e_{r}\left(q^{-\lambda_1}t^{1-{n}},\dots,q^{-\lambda_{n}}\right)|\le q^{-|\lambda|} t^{-r} n^r.$$
 Therefore, the sum in the left side of \eqref{eq_observable_with_vars_inversed} is absolutely
 bounded by
 \begin{align*}
& {\rm const}\cdot  \sum_{\lambda^1,\dots,\lambda^N\in\Y}
   q^{-m|\lambda^N|}\left|\MP_{N;\{a_i\};\rho}(\lambda^1,\dots,\lambda^N)\right|
 \\&= {\rm const} \cdot  \sum_{\lambda^1,\dots,\lambda^N\in\Y}
 \left|\MP_{N;\{a_i\};q^{-m}\rho}(\lambda^1,\dots,\lambda^N)\right|,
 \end{align*}
 where $q^{-m} \rho$ is the specialization defined by
 $$
  p_k(q^{-m}\rho)=q^{-km} p_k(\rho),
 $$
 which, by the hypothesis of the theorem (that $|p_k(\rho)|<R^k$) implies that $|p_{k}(q^{-m}\rho)|<(q^{-m}R)^k$. Therefore
 $$
 |(a_i)^k p_k(q^{-m}\rho)|<r^k
 $$
 for some $r<1$.
 Now the absolute convergence of the series in the Definition \ref{def_Mac_ascending_proc}
 implies that the series in the left side of \eqref{eq_observable_with_vars_inversed} is absolutely
 convergent.

The definition of Macdonald symmetric functions $P_\lambda$ implies that they are invariant under the change of parameters
 $(q,t)\leftrightarrow(q^{-1},t^{-1})$ (see \cite[Section VI.4, (4.14),(iv)]{M}). In other words,
 restoring the notational dependence of $P_\lambda$ on $(q,t)$, we have
 $$
  P_\lambda(\,\cdot\,;q,t) = P_\lambda(\,\cdot\,; q^{-1},t^{-1}).
 $$
 Therefore, if we replace all instances of $q$ and $t$ in the definition of the operator $\Mac^r_k$ by $q^{-1}$
 and $t^{-1}$, respectively, and denote the resulting operator through $\widehat \Mac^r_k$, then
$$
  \widehat \Mac^r_kP_\lambda(x_1,\dots,x_k;q,t) = e_r(q^{-\lambda_1}t^{1-k},\dots,q^{-\lambda_k})
  P_\lambda(x_1,\dots,x_k;q,t).
$$
It follows that an analogue of Proposition \ref{prop_observable_as_application_of_op} holds for
$\widehat \Mac^r_k$ and we can repeat the proof of Theorem \ref{theorem_ascending_multilevel_1} (as
long as all the series involved converge). The final formula
\eqref{eq_observable_with_vars_inversed} is obtained from the result of Theorem
\ref{theorem_ascending_multilevel_1} through the formal inversion of $q$ and $t$.
\end{proof}

One interesting limit of the above formulas can be obtained by sending $t\to 0$. The limits of
Macdonald symmetric functions themselves as $t\to 0$ are known as $q$--Whittaker functions, cf.\
\cite{GLO}. Denote
$$
  \WP_{N;\{a_i\};\rho}=\lim_{t\to 0} \MP_{N;\{a_i\};\rho}.
$$

A straightforward limit of (the case $r_i=1$, $i=1,\dots,m$ of) Theorem
\ref{theorem_ascending_multilevel_1} gives the following statement.
\begin{corollary}
\label{Corollary_q_TASEP}
 Take $N$ non--zero complex numbers $a_1$,\dots, $a_N$ and a specialization $\rho$ such that
  for some $0<R<1$ we have
 $$
  |p_k(\rho)|<R^k,\ k=1,2,\dots, \quad \quad |a_i|R<1,\, i=1,\dots,N.
 $$
 Fix a sequence of integers $N\ge n_1 \ge n_2\ge\dots\ge n_1\ge 1$.
 Suppose that there exist complex closed positively oriented contours
 $\gamma_{\alpha}$, $\alpha=1,\dots,m$, such that
 \begin{enumerate}[(I)]
 \item All contours lie inside $\mathbb D_{R^{-1}}=\{z\in\mathbb C\mid |z|<R^{-1}\}$.
\item All contours enclose the points $a_1,\dots,a_N$, but not $0$,
\item  $\gamma_{\alpha}$ contour does not intersect the interior of the image of $\gamma_{\beta}$ multiplied by $q$  for
 $\beta>\alpha$.
 \end{enumerate}
Then we have
\begin{multline}
\label{eq_x22}
  \sum_{\lambda^1,\dots,\lambda^N\in\Y} \prod_{i=1}^m
  q^{\lambda^{n_i}_{n_i}} \WP_{N;\{a_i\};\rho}(\lambda^1,\dots,\lambda^N)
  \\= \frac{(-1)^m q^{m(m-1)/2}}{(2\pi \i)^m} \oint\dots\oint
  \prod_{1\le \alpha <\beta \le m}
  \frac{z_{\alpha}-z_{\beta}}{z_{\alpha}-q z_{\beta}}
   \prod_{\alpha=1}^m \left(
  \prod_{i=1}^{n_\alpha} \frac{a_i}{a_i-z_{\alpha}}\right)
  \frac{\widetilde \Pi(qz_{\alpha};\rho)}{\widetilde \Pi(z_{\alpha};\rho)}\frac{dz_{\alpha}}{z_{\alpha}}
 \end{multline}
 where $z_{\alpha}$ is integrated over $\gamma_\alpha$ and
 $$
  \widetilde\Pi(z;\rho)=\lim_{t\to 0} \Pi(z;\rho)=\exp\left(\sum_{k=1}^{\infty} \frac{z^k
  p_k(\rho)}{k(1-q^k)}\right).
 $$
 \end{corollary}
{\bf Remark.} For the choice of $\rho$ such that $\Pi(z,\rho)=\exp(\tau z)$ for a parameter
$\tau>0$, and $a_1=\dots=a_N=1$, the formula \eqref{eq_x22} was guessed and checked in \cite{BCS}.
The formulas of \cite{BC_discrete} are also related to some particular choices of $\rho$ in
\eqref{eq_x22}.

\subsection{Comparison with formal setting}
Let us compare the statements of Theorems \ref{theorem_ascending_multilevel_1} and
\ref{theorem_observable_with_vars_inversed} with formal statements of Section
\ref{Section_Formal_Macdonald}.

 The statement of
 Theorem \ref{theorem_observable_with_vars_inversed} can be
obtained from Theorem \ref{theorem_observable_multilevel} by first specializing algebras
$\Lambda_{A^i}$, $\Lambda_{B^i}$ (see Definition \ref{def_Mac_ascending_proc}), then changing the
variables $z=1/w$ and further suitably deforming  the contours of integration. Note that the
contours in Theorem \ref{theorem_observable_with_vars_inversed} do not enclose $0$, while in
Theorem \ref{theorem_observable_multilevel} they do. Thus, we would pick certain residues while
deforming the contours, and these residues are responsible for the change in the observable. In
fact, for $m=1$ we performed a similar deformation in the proof of Lemma
\ref{proposition_restated_computation_from_BigMac}.

A formal version of Theorem \ref{theorem_ascending_multilevel_1} is more delicate. The difficulty
lies in the fact that our observable $e_r(q^{\lambda_1}t^{k-1},\dots,q^{\lambda_k})$ does not have
a straightforward limit as $k\to\infty$. In the appendix (Section
\ref{section_Different_form_first_operator}) we give a formal version of Theorem
\ref{theorem_ascending_multilevel_1} for the case $r_i=1$ for all $i$. Finding such a formal
version for general $r_i$ involves finding a suitable form of \emph{stable} Macdonald operators and
we do not pursue this here.

\subsection{Fredholm determinants}\label{freddetsec}
The aim of this section is to present two observables of the Macdonald measure whose expectations
can be written as Fredholm determinants. For the case $q=t$ (corresponding to Schur polynomials)
the emergence of Fredholm determinants is well--understood, due to the identification of the Schur
measure with a determinantal point process (see \cite{Ok} and also \cite{BG} for a recent review).
No such structure predicting the appearance of Fredholm determinants is known for general
parameters $(q,t)$.  The $t=0$ and general $q$ degeneration of the first Fredholm determinant we
present in Theorem \ref{Theorem_qtdet} was discovered in \cite[Corollary 3.2.10 and Theorem
3.2.11]{BigMac} by utilizing the first Macdonald difference operator and its powers. Our present
result uses a different operator diagonalized by the Macdonald polynomials (that however
degenerates to a generating series of powers of the first Macdonald operator at $t=0$).

We will write Fredholm determinant formulas for the (ascending) Macdonald measure
$\MM_{N,\{a_i\};\rho}$. We do not approach the question of generalizing the formulas below to the formal settings
of Section \ref{Section_Formal_Macdonald} (again this is known to be possible for $q=t$ case).

\medskip

\begin{theorem}\label{Theorem_qtdet}

Fix $N$ non--zero complex numbers $a_1$,\dots, $a_N$ and a specialization $\rho$ such that for some
$0<R<1$
 we have:
 $$
  |p_k(\rho)|<R^k,\ k=1,2,\dots, \quad \quad |a_i|R<1,\, i=1,\dots,N.
 $$
Define the function
\begin{equation*}
\G(w) = \frac{1}{\Pi(w;\rho)} \prod_{j=1}^{N} \frac{(tw/a_j;q)_{\infty}}{(w/a_j;q)_{\infty}}.
\end{equation*}
Further, let $u$ be a formal variable and set
\begin{equation*}
\K(w,w') =  \sum_{v=1}^{\infty} \frac{u^v}{q^v w-w'} \frac{\G(w)}{\G(q^v w)}.
\end{equation*}
Suppose that there exists a positively oriented contour $\gamma$ such that:
\begin{enumerate}[(I)]
\item $\gamma$ lies inside $\mathbb D_{R^{-1}}=\{z\in\mathbb C\mid |z|<R^{-1}\}$,
\item $\gamma$ encloses all points $a_i$, $i=1,\dots,N$, but not $q^s t a_i$, $s=1,2,\dots$,
$i=1,\dots,N$.
\item the contour $q^k\gamma$ is outside $\gamma$ for $k=1,2,\dots$.
\end{enumerate}
(Note that when $a_1=\dots=a_N$, a small circle around $a_1$ satisfies all the above assumptions.)

Then the following equality holds as an identity of power series in $u$:
\begin{equation*}
\sum_{\lambda\in\Y} \prod_{i=1}^{N}
\frac{\big(q^{\lambda_i}t^{N-i+1}u;q\big)_{\infty}}{\big(q^{\lambda_i}t^{N-i}u;q\big)_{\infty}}\,
\MM_{N,\{a_i\},\rho}(\lambda) = \det(I+\K)_{L^2(\gamma)}.
\end{equation*}
\end{theorem}

{\bf Remark 1.} In the theorem we \emph{define} the Fredholm determinant $\det(I+\K)_{L^2(\gamma)}$ through its Fredholm series
expansion
\begin{equation}
\label{eq_Fredholm_definition} \det(I+\K)_{L^2(\gamma)} = 1+ \sum_{k=1}^{\infty} \frac{1}{k!}
\int_{\gamma} dw_1 \cdots \int_{\gamma} dw_k \det\left[\K(w_i,w_j)\right]_{i,j=1}^{k}.
\end{equation}
Our definitions imply that $k$th term in the sum in \eqref{eq_Fredholm_definition} is a power
series in $u$ starting from $u^k$. Therefore, \eqref{eq_Fredholm_definition} is a well--defined
power series in $u$. In fact, this power series is easily seen in the proof to be a degree $N$ polynomial in $u$.

{\bf Remark 2.} When the above theorem is seen as analytic identity
 after specifying some value of $u$ for which all the series absolutely converge,
 then the kernel $\K$ can be represented in the following form:
\begin{equation*}
\K(w,w') =  \int_{C_{1,2,\ldots}}\Gamma(-s)\Gamma(1+s) \frac{(-u)^s}{q^s w-w'} \frac{\G(w)}{\G(q^s
w)}
\end{equation*}
where $C_{1,2,\ldots}$ is a negatively oriented contour which encloses the positive integer poles
of $\Gamma(-s)\Gamma(1+s)$ and no other poles of the integrand. This representation is useful when
performing asymptotics (cf.\ \cite[Section 3.2.3]{BigMac}, \cite[Sections 5 and 6]{BCF}).

{\bf Remark 3.} When $t=0$, \cite[Theorem 3.2.16]{BigMac} provides a second Fredholm determinant
formula for the same expectation which differs from the result of taking $t=0$ in the above theorem
(cf. \cite[Corollary 3.2.10 and Theorem 3.2.11]{BigMac}). It might be possible to write down an
analog to this second type of Fredholm determinant at the general $(q,t)$ level, though we do not
pursue that here as it is so far unclear if it has applications.

\bigskip

We now present the proof of Theorem \ref{Theorem_qtdet}. At the end of the section we state one
other Fredholm determinant result in the form of Theorem \ref{ekfred}.

\begin{proof}[Proof of Theorem \ref{Theorem_qtdet}]
The following operator and eigenfunction relation in Proposition \ref{NoumiProp} come from
\cite[Proposition 3.24]{FHHSY} and therein is attributed to a personal communication from M.~Noumi;
its proof will appear in \cite{Noumi}. We also present E.~Rains' proof of this result as an
appendix in Section \ref{app}.

\begin{definition}\label{NoumiOpDef}
For any $N\geq 1$  and any $\nu\in \N^N$ define a difference operator $\Noumi{u}{N,\nu}$ which
acts on the space of analytic functions in $x_1,\ldots, x_N$ as
\begin{equation}\label{NoumiOpEqnNu}
\Noumi{u}{N,\nu} = u^{|\nu|} \prod_{1\le i<j\le N} \frac{q^{\nu_j}x_j-q^{\nu_i}x_i}{x_j-x_i}
\prod_{1\le i,j\le N} \frac{(tx_i/x_j;q)_{\nu_i}}{(qx_i/x_j;q)_{\nu_i}} \prod_{i=1}^{N}
\big(T_{q,i}\big)^{\nu_i}.
\end{equation}
Define the \emph{Noumi $q$--integral operator} $\Noumi{u}{N}$  which acts on the space of
analytic functions in $x_1,\ldots, x_N$ as
\begin{equation}\label{NoumiOpEqn}
\Noumi{u}{N} = \sum_{\nu\in\N^N} \Noumi{u}{N,\nu}.
\end{equation}
\end{definition}

The Macdonald polynomials diagonalize the Noumi $q$--integral operator with explicit eigenvalues.

\begin{proposition}\label{NoumiProp}
For any $N\geq 1$, formal parameter $u$, and $\lambda\in \Y$ such that $\ell(\lambda)\le N$, the
following identity of power series in $u$ holds
\begin{equation}\label{NoumiEigRel}
\Noumi{u}{N} P_{\lambda}(x_1,\ldots, x_n) = \prod_{1\le i\le N}
\frac{(q^{\lambda_i}t^{N+1-i}u;q)_{\infty}}{(q^{\lambda_i}t^{N-i}u;q)_{\infty}}P_{\lambda}(x_1,\dots,x_N).
\end{equation}
\end{proposition}
This proposition is proved in Section \ref{app}.

{\bf Remark.} \label{ellseries} For $r\geq 0$, the operator
\begin{equation*}
\Noumi{[r]}{N}\sum_{\substack{\nu\in \N^N\\|\nu|=r}} \Noumi{u}{N,\nu}
\end{equation*}
is also diagonalized by the $P_{\lambda}$ with eigenvalues $g_{r}(q^{\lambda_1}t^{N-1},
q^{\lambda_2}t^{N-2},\ldots, q^{\lambda_N}t^0)$. Here $g_{r}$ is the $(q,t)$--version of the
complete homogeneous symmetric polynomial (i.e., $g_r=Q_{(r)}$). Clearly
$\Noumi{u}{N}=\sum_{r=0}^{\infty} u^r \Noumi{[r]}{N}.$

%The usefulness of this proposition to the study of Macdonald processes (and the proof of Theorem
%\ref{Theorem_qtdet} is explained by the following basic observation. Assume we have a linear
%operator $\mathcal{D}$ in the space of functions in $n$ variables whose restriction to the space of
%symmetric polynomials diagonalizes in the basis of Macdonald polynomials: $\mathcal{D}
%P_\lambda=d_\lambda P_\lambda$ for any partition $\lambda$ with $\ell(\lambda)\le n$. Then we can
%apply $\mathcal{D}$ to both sides of the identity
%\begin{equation*}
%\sum_{\lambda:\ell(\lambda)\le n} P_{\lambda}(x_1,\dots,x_n)
%Q_{\lambda}(\rho)=\Pi(x_1,\dots,x_n;\rho).
%\end{equation*}
%Dividing the result by $\Pi(x_1,\dots,x_n;\rho)$ we obtain
%\begin{equation}\label{tag8}
%\langle d_\lambda \rangle_{\MM(x_1,\dots,x_n;\rho)}=\frac{\mathcal{D}\Pi(x_1,\dots,x_n;\rho)}
%{\Pi(x_1,\dots,x_n;\rho)}\,,
%\end{equation}
%where $\langle \cdot\rangle_{\MM(x_1,\dots,x_n;\rho)}$ represents averaging $\cdot$ over the
%specified Macdonald measure.
%
%\note{mention the difference operators and the use in the other section?}

\medskip

Proposition \ref{NoumiProp} implies that
\begin{equation}\label{lreqn}
\sum_{\lambda \in \Y} \prod_{1\le i\le N}
\frac{(q^{\lambda_i}t^{N+1-i}u;q)_{\infty}}{(q^{\lambda_i}t^{N-i}u;q)_{\infty}}\,
\MM_{N,\{a_i\};\rho}(\lambda) = \frac{\Noumi{u}{N} \Pi(x_1,\ldots,
x_N;\rho)}{\Pi(x_1,\ldots,x_N;\rho)}\Bigg|_{x_1=a_1,\dots,x_N=a_N},
\end{equation}
the argument here being parallel to that of Proposition \ref{prop_observable_as_application_of_op},
see also \cite[Section 2.2.3]{BigMac} for a general discussion. The only thing to check here is
that the series giving the coefficient of $u^r$ in the left side of \eqref{lreqn} is absolutely
convergent. This series is (cf.\ \cite[Chapter VI, (2.8)]{M})
\begin{equation}
\label{eq_x32} \sum_{\lambda \in \Y} g_{r}(q^{\lambda_1}t^{N-1}, q^{\lambda_2}t^{N-2},\ldots,
q^{\lambda_N}t^0) \, \MM_{N,\{a_i\};\rho}(\lambda).
\end{equation}
The combinatorial formula for Macdonald polynomials $g_r$ (see \cite[Chapter VI, Section 7]{M}) and
inequalities $0<q,t<1$ imply that
$$0\le g_{r}(q^{\lambda_1}t^{N-1}, q^{\lambda_2}t^{N-2},\ldots,
q^{\lambda_N}t^0)\le g_r(\underbrace{1,\dots,1}_N).$$ Thus, \eqref{eq_x32} is absolutely convergent
as in Definition \ref{def_Mac_ascending_proc}.

Recall (cf.\ Section \ref{formmacproc}) that $\Pi(x_1,\ldots,x_N;\rho) =\Pi(x_1;\rho)\cdots
\Pi(x_N;\rho)$. For such functions, it is possible to encode the application of the Noumi operator
in terms of a Fredholm determinant. Theorem \ref{Theorem_qtdet} immediately follows from equation
\eqref{lreqn} along with the application of the following proposition.

\begin{proposition}
\label{Proposition_qtFredholm1} The following holds as an identity of power series in $u$:
\begin{equation*}
 \frac{\Noumi{u}{N} \Pi(x_1,\ldots,
x_N;\rho)}{\Pi(x_1,\ldots,x_N;\rho)}\Bigg|_{x_1=a_1,\dots,x_N=a_N} = \det(I+\K)_{L^2(\gamma)}.
\end{equation*}
%where $C_x$ is a contour enclosing the points $\{x_1,\ldots,x_n\}$ and no other poles of the kernel
%(i.e., small circles around the $x_i$) and where the Fredholm determinant is given by its expansion
%and the associated integrals are interpreted in terms of residues \note{this should be explained
%better}. Here
%\begin{equation}
%K(w,w') =  \sum_{v=1}^{\infty} \frac{u^v}{q^v w-w'} \frac{G(w)}{G(q^v w)}
%\end{equation}
%and
%\begin{equation}
%G(w) = \frac{1}{f(w)} \prod_{j=1}^{n} \frac{(tw/x_j;q)_{\infty}}{(w/x_j;q)_{\infty}}.
%\end{equation}
\end{proposition}
\begin{proof}

We proceed in three steps. In step 1 we show how simple residue considerations imply that the
Fredholm expansion for $\det(I+\K)_{L^2(\gamma)}$ terminates after $N$ terms. In step 2 we present
a lemma which relates the $k^{th}$ term in this expansion to the application of the operators
$\Noumi{u}{N,\nu}$ with the number of non--zero parts of $\nu$ equal to $k$. In step 3 we conclude
the proof by combining the two previous steps.

\noindent {\bf Step 1:} Recall the definition of $\det(I+\K)_{L^2(\gamma)}$ via the Fredholm series expansion.
We can rewrite
\begin{equation*}
\K(w,w') = \prod_{r=1}^{N} \frac{1}{w-a_r}  \Ktilde(w,w')
\end{equation*}
where $\Ktilde(w,w')$ is now analytic inside the contour $\gamma$ in both variables. This means
that we can evaluate all of the $w_i$ integrations in \eqref{eq_Fredholm_definition} via the
residue theorem. Each variable $w_i$ can pick a residue at any of $\{a_1,\ldots, a_N\}$. This leads
to the expansion
\begin{equation*}
\int_{\gamma} dw_1 \cdots \int_{\gamma} dw_k \det\left[\K(w_i,w_j)\right]_{i,j=1}^{k} = \sum_{p}
\det\left[ \prod_{\substack{r=1\\r\neq p(i)}}^{N} \frac{1}{a_{p(i)}-a_r}
\tilde{K}_u(a_{p(i)},a_{p(j)})\right]_{i,j=1}^{k}
\end{equation*}
where the summation is over all assignments $p:\{1,\ldots k\}\to \{1,\ldots, N\}$. If $k>N$ then
there must exist some $i\neq i'$ such that $p(i)=p(i')$. Consequently, row $i$ and row $i'$ of the
above matrix coincide, hence the determinant is zero. Thus
\begin{equation}\label{deteqn2}
\det(I+\K)_{L^2(\gamma)} = 1+ \sum_{k=1}^{N} \frac{1}{k!} \int_{\gamma} dw_1 \cdots \int_{\gamma} dw_k
\det\left[\K(w_i,w_j)\right]_{i,j=1}^{k}.
\end{equation}

\noindent {\bf Step 2:} We now show how the $k^{th}$ term in equation \eqref{deteqn2} arises from
a combination of the $\Noumi{u}{N,\nu}$ with the number of non--zero parts of $\nu$ equal to $k$
(and these non--zero parts summed over the natural numbers).
\begin{lemma}\label{step2lem}
Fix $N\geq 0$, $k\in \{0,1,\ldots,N\}$ and assume $\nu\in \N^N$ is such that $\nu_1,\ldots,
\nu_k\geq 1$ and $\nu_{k+1},\ldots, \nu_{N} =0$. Then, for all $a_1,\ldots, a_N$ and $u$,
\begin{multline}\label{lemmasneqn}
\frac{1}{(N-k)!}  \sum_{\sigma\in S_N}
\frac{\Noumi{u}{N,\sigma(\nu)}\Pi(x_1,\ldots, x_N;\rho)}{\Pi(x_1,\ldots,
x_N;\rho)}\Bigg|_{x_1=a_1,\dots,x_N=a_N}  \\
 = \frac{1}{(2\pi \i)^k} \int_{\gamma} dw_1 \cdots \int_{\gamma} dw_k
\det\left[\Kprime(\nu_i,w_i,w_j)\right]_{i,j=1}^{k}
\end{multline}
where
\begin{equation*}
\Kprime(v,w,w') = \frac{u^{v}}{q^{v} w-w'} \frac{\G(w)}{\G(q^{v} w)},
\end{equation*}
and $S_n$ is the symmetric group of rank $N$ which acts on $\nu=(\nu_1,\ldots, \nu_k)$ by
permuting its coordinates.
\end{lemma}
\begin{proof}
We evaluate the right--hand side of equation \eqref{lemmasneqn} via residues in order to prove the
theorem. Observe that the only term in $\Kprime$ involving both $w$ and $w'$ is $(q^v w-w')^{-1}$. The
Cauchy determinant identity  \eqref{eq_Cauchy_det} may be applied to this term, and a small
calculation and reordering of terms leads to
\begin{eqnarray*}
\textrm{RHS \eqref{lemmasneqn}} &=& \frac{1}{(2\pi \i)^k}\int_{\gamma} dw_1 \cdots \int_{\gamma} dw_k
\prod_{i=1}^{k} \prod_{j=1}^{N} \frac{-1}{w_i-a_j}
 \\ & &\times \prod_{\substack{i,j=1\\i\neq j}}^{k} (w_i-w_j) \prod_{i,j=1}^{k} \frac{-1}{w_j-q^{\nu_i}w_i} \prod_{i=1}^{k}\prod_{j=1}^{N} (a_j-q^{\nu_i}w_i)\\
&& \times u^{|\nu|} \prod_{1\leq i<j\leq k} \frac{q^{\nu_i}w_i - q^{\nu_j}w_j}{w_i-w_j}
\prod_{i=1}^{k} \prod_{j=1}^{N} \frac{(tw_i/a_j;q)_{\infty}}{(qw_i/a_j;q)_{\infty}}
\prod_{i=1}^{k} \frac{\Pi(q^{\nu_i}w_i;\rho)}{\Pi(w_i;\rho)}.
\end{eqnarray*}
Inspection of the above formula reveals that it is only the first product which has poles inside
the $\gamma$ contour. The residue theorem implies that we can evaluate the above integral by
computing the sum of the residues at $w_i=a_{p(i)}$ $1\leq i\leq k$, summed over every choice of
assignment $p:\{1,\ldots, k\}\to \{1,\ldots, N\}$. Further inspection reveals that due to factors
$(w_i-w_j)$, if $p(i)=p(j)$ for some $i\neq j$, then the residue is zero. Hence we are left with
the sum over all assignments for which $p(i)\neq p(j)$ when $i\neq j$. As a convention, define
$p(1+k),\ldots, p(N)$ to be the (ordered) remaining elements of $\{1,\ldots, N\}$ which are not
equal to $p(1),\ldots, p(k)$. Denote as $P$ the set of all such defined assignment (or
permutations) from $\{1,\ldots, N\}\to \{1,\ldots, N\}$. Thus (after noticing that various factors
of $-1$ multiply to $1$)
\begin{eqnarray*}
\textrm{RHS \eqref{lemmasneqn}} &=& \sum_{p\in P} \prod_{i=1}^{k} \prod_{j=k+1}^{N} \frac{q^{\nu_i}a_{p(i)} - a_{p(j)}}{a_{p(i)}-a_{p(j)}}  \\
&& \times u^{|\nu|} \prod_{1\leq i<j\leq k} \frac{q^{\nu_i}x_{p(i)} -
q^{\nu_j}a_{p(j)}}{a_{p(i)}-a_{p(j)}} \prod_{i=1}^{k} \prod_{j=1}^{N}
\frac{(ta_{p(i)}/a_{p(j)};q)_{\infty}}{(qa_{p(i)}/a_{p(j)};q)_{\infty}} \prod_{i=1}^{k}
\frac{\Pi(q^{\nu_i}a_{p(i)};\rho)}{\Pi(a_{p(i)};\rho)}.
\end{eqnarray*}
 Recalling that
$\nu_{k+1}=\ldots=\nu_N=0$ we can combine the above expressions as
\begin{equation*}
\textrm{RHS \eqref{lemmasneqn}} = \sum_{p\in P} u^{|\nu|} \prod_{1\leq i<j\leq N}
\frac{q^{\nu_i}a_{p(i)} - q^{\nu_j}a_{p(j)}}{a_{p(i)}-a_{p(j)}} \prod_{i,j=1}^{N}
\frac{(ta_{p(i)}/a_{p(j)};q)_{\infty}}{(qa_{p(i)}/a_{p(j)};q)_{\infty}} \prod_{i=1}^{N}
\frac{\Pi(q^{\nu_i}a_{p(i)};\rho)}{\Pi(a_{p(i)};\rho)}.
\end{equation*}
Again, due to the fact that $\nu_{k+1}=\ldots=\nu_N=0$, the summation over $p\in P$ can be replaced
by $p\in S_N$, yielding
\begin{equation*}
\textrm{RHS \eqref{lemmasneqn}} = \frac{1}{(N-k)!} \sum_{p\in S_N} u^{|\nu|} \prod_{1\leq i<j\leq
N} \frac{q^{\nu_i}a_{p(i)} - q^{\nu_j}a_{p(j)}}{a_{p(i)}-a_{p(j)}} \prod_{i,j=1}^{N}
\frac{(ta_{p(i)}/a_{p(j)};q)_{\infty}}{(qa_{p(i)}/a_{p(j)};q)_{\infty}} \prod_{i=1}^{N}
\frac{\Pi(q^{\nu_i}a_{p(i)};\rho)}{\Pi(a_{p(i)};\rho)}.
\end{equation*}
The $(N-k)!$ came from the size of $S_N/P$. We now call $\sigma = p^{-1}$ and replace $a_{p(i)}$ by
$a_i$ and $\nu_i$ by $\nu_{\sigma(i)}$ in the above expression. Noting that
\begin{equation*}
\frac{(T_{q,i}^{\nu_i}\Pi)(x_1,\ldots, x_N;\rho)}{\Pi(x_1,\ldots, x_N;\rho)} =
\frac{\Pi(q^{\nu_i}x_{i};\rho)}{\Pi(x_{i};\rho)},
\end{equation*}
we are finally led to
\begin{equation*}
\textrm{RHS \eqref{lemmasneqn}} = \frac{1}{(N-k)!}  \sum_{\sigma\in S_N}
\frac{\Noumi{u}{N,\sigma(\nu)} \Pi(x_1,\ldots, x_N;\rho)}{\Pi(x_1,\ldots,
x_n;\rho)}\Bigg|_{x_1=a_1,\dots,x_N=a_N},
\end{equation*}
as desired to prove the lemma.
\end{proof}

\noindent {\bf Step 3:} We now rewrite the Noumi $q$--integral operator in terms of the expressions
on the left--hand side of equation \eqref{lemmasneqn}. In particular, we split the summation
defining $\Noumi{u}{N}$ based on the number of non--zero parts to $\nu$:
\begin{eqnarray*}
\Noumi{u}{N}= \sum_{\nu\in \N^N} \Noumi{u}{N,\nu} &=& \sum_{k=0}^{N} \sum_{\substack{S\subseteq\{1,\ldots, N\}\\|S|=k}} \sum_{\substack{\nu\in \N^N\\ \nu_{i}>0 \textrm{ for } i\in S \nu_{i}=0 \textrm{ for } i\notin S\\}} \Noumi{u}{N,\nu}\\
&=& \sum_{k=0}^{N} \sum_{\nu_1,\ldots \nu_k=1}^{\infty} \frac{N!}{(N-k)! k!} \frac{1}{N!}
\sum_{\sigma\in S_N} \Noumi{u}{N,\sigma(\nu)}
\end{eqnarray*}
where in the second line $\nu = (\nu_1,\ldots, \nu_k,0,\ldots 0)$. Note that the $N$ choose $k$ factor came from the number of ways of choosing the subset $S$, and the reciprocal of $N$ factorial came from the symmetrization of $\nu$.

Using the above calculation and Lemma \ref{step2lem}, we find that

\begin{align*}
&\frac{\Noumi{u}{N}\Pi(x_1,\ldots, x_N;\rho)}{\Pi(x_1,\ldots,
x_n;\rho)}\Bigg|_{x_1=a_1,\dots,x_N=a_N}\\
&=\sum_{k=0}^{N} \frac{1}{k!} \sum_{\nu_1,\ldots \nu_k=1}^{\infty}
\frac{1}{(2\pi \i)^k} \int_{\gamma} dw_1 \cdots \int_{\gamma} dw_k  \det\left[\Kprime(\nu_i,w_i,w_j)\right]_{i,j=1}^{k}\\
&= \sum_{k=0}^{N} \frac{1}{k!} \frac{1}{(2\pi \i)^k} \int_{\gamma} dw_1 \cdots \int_{\gamma} dw_k
\det\left[\K(w_i,w_j)\right]_{i,j=1}^{k}
\end{align*}
where in the third line the summation over the $\nu_i$ was absorbed into the determinant
(resulting in the $\K$ kernel). Finally, by virtue of equation \eqref{deteqn2} from step 1, we
conclude the proof of the proposition.
\end{proof}

As explained before the statement of Proposition \ref{Proposition_qtFredholm1}, this also completes the proof of Theorem \ref{Theorem_qtdet}.
\end{proof}

We present a second general $(q,t)$ Fredholm determinant which relies upon the Macdonald difference
operators $\Mac^r_N$ (see Section \ref{Section_ascending_obs}) and their elementary symmetric
function eigenvalues (see equation \eqref{eq_Macdonald_operator_eigen}).

Letting $u$ be a formal parameter, we may define
\begin{equation*}
\Mac_N(u) = \sum_{r=0}^{N} (-u)^r \Mac^r_N.
\end{equation*}
Then, for $\lambda\in \Y$ with $\ell(\lambda)\leq N$, we have
\begin{equation}\label{macurel}
\Mac_N(u) P_{\lambda}(x_1,\ldots,x_N) = \prod_{i=1}^{N} (1-u q^{\lambda_i} t^{N-i}).
\end{equation}
This follows from equation  \eqref{eq_Macdonald_operator_eigen} since the right--hand side above
is the generating functions for the elementary symmetric polynomials, cf.\ \cite[Chapter I,
Section 2]{M}.

\begin{theorem}\label{ekfred}
Fix $N$ non--zero complex numbers $a_1$,\dots, $a_N$ and a specialization $\rho$ such that for some
$0<R<1$
 we have:
 $$
  |p_k(\rho)|<R^k,\ k=1,2,\dots, \quad \quad |a_i|R<1,\, i=1,\dots,N.
 $$
Let $C_a$ be a contour which  lies inside a circle of radius $R^{-1}$ and which encloses all $a_1,\ldots, a_N$
but not $ta_1,\ldots, ta_N$. Set
\begin{equation*}
\J(w,w') = \frac{1}{tw'-w} \prod_{m=1}^{N}\frac{tw-a_m}{w-a_m} \, \frac{\Pi(qw;\rho)}{\Pi(w;\rho)}.
\end{equation*}

Then the following equality holds as an identity of power series in $u$
\begin{equation*}
\sum_{\lambda\in\Y} \prod_{i=1}^{N} (1-u q^{\lambda_i} t^{N-i})\MM_{N,\{a_i\},\rho}(\lambda) =
\det(I-u\J)_{L^2(C_a)}.
\end{equation*}
\end{theorem}
{\bf Remark 1.} The above Fredholm determinant is related to the generating function of the
elementary symmetric polynomials $e_r$ whereas the Fredholm determinant presented in Theorem
\ref{Theorem_qtdet} is related (see Proposition \ref{NoumiProp}) to the $(q,t)$--analog of the
complete homogeneous symmetric polynomials $g_{r}$. There is an endomorphism $\omega_{q,t}$ on
$\Lambda$ which maps $\omega_{q,t} g_{r}(X;q,t) = e_r(X)$. At this point, it is not clear how this
endomorphism is related to the two Fredholm determinant formulas we have presented.

{\bf Remark 2.} In equation (3.3) of \cite{Warnaar} (for $\rho$ a finite length specialization into
a set of complex numbers $y_1,y_2,\ldots$) an alternative expression (written as $F(u;x,y;t)$) is
given for the above Fredholm determinant. This function is then related to the Izergin--Korepin
determinant.

{\bf Remark 3.} It is possible to state a formal version of the above theorem immediately from
Proposition \ref{proposition_restated_computation_from_BigMac}.

\begin{proof}[Proof of Theorem \ref{ekfred}]
Observe that by virtue of the eigenrelation \eqref{macurel} satisfied by $\Mac_N(u)$,
\begin{equation*}
\frac{\Mac_N(u)\Pi(x_1,\ldots, x_N;\rho)}{\Pi(x_1,\ldots,
x_N;\rho)}\Bigg|_{x_1=a_1,\dots,x_N=a_N} = \sum_{\lambda\in\Y} \prod_{i=1}^{N}
(1-u q^{\lambda_i} t^{N-i})\MM_{N,\{a_i\},\rho}(\lambda).
\end{equation*}

A special case of Theorem \ref{theorem_ascending_multilevel_1} (also found in \cite[Proposition 2.2.10]{BigMac}) states that
\begin{equation*}
\frac{\Mac^r_N\Pi(x_1,\ldots, x_N;\rho)}{\Pi(x_1,\ldots,
x_N;\rho)}\Bigg|_{x_1=a_1,\dots,x_N=a_N} =  \frac{1}{r!} \frac{1}{(2\pi \i)^r} \oint_{C_a}dw_1\cdots \oint_{C_a} dw_r \det\left[\J(w_k,w_\ell)\right]_{k,\ell=1}^{r}.
\end{equation*}

Multiplying each term by $(-u)^r$ and summing over $r\in \{0,1,\ldots, N\}$ we recover the first
$N+1$ terms in the Fredholm series expansion of $\det(I-u\J)_{L^2(C_a)}$. It is easy to see that
all further terms in the expansion vanish (this is somewhat similar to step 1 in the proof of
Theorem \ref{Theorem_qtdet}), hence the desired result.
\end{proof}

\section{Appendix: E.~Rains' proof of Proposition \ref{NoumiProp}}\label{app}

The following appendix, due to Eric Rains, provides a derivation of Proposition \ref{NoumiProp}
from an elliptic integral operator. Let us fix some notation:
$$
\theta_q(x) := \prod_{k\geq 0} (1-q^k x)(1-q^{k+1}/x), \qquad
\Gamma_{p,q}(x) := \prod_{j,k\geq 0} \frac{1-p^{j+1}q^{k+1}/x}{1-p^jq^k x}, \qquad
\tilde{\Gamma}_{q}(x) := \prod_{k\geq 0} \frac{1}{1-q^k x}.
$$
Note that $\tilde{\Gamma}_q(x)$ is slightly different than the usual definition of the $q$--deformed
Gamma function, hence the tilde. When multiple arguments come into these functions, it means that
one multiplies the single variable evaluation over all variables. For example,
\begin{equation*}
\Gamma_{p,q}(y_i^{\pm}y_j^{\pm}) := \Gamma_{p,q}(y_iy_j)\Gamma_{p,q}(y_iy_j^{-1})
\Gamma_{p,q}(y_i^{-1}y_j)\Gamma_{p,q}(y_i^{-1} y_j^{-1}),
\end{equation*}
or $(a,b;q)_{\infty}=(a;q)_\infty (b;q)_\infty$. In what follows a pair of partitions is denoted
by a bold lambda $\blambda$ whereas a single partition is just $\lambda$.

The elliptic interpolation functions $\cR^{*(n)}_{\blambda}(y_1,\dots,y_n;u_0,u_1;t;p,q)$ are
defined in equation (8.45) of \cite{RainsTrans}. They satisfy the following integral operator
identity, which is itself a special case of that given in equation (8.12) of \cite{RainsTrans}:
\begin{multline*}
\frac{\cR^{*(n)}_{\blambda}(y_1,\dots,y_n;u_0,u_1;t;p,q)}
     {\cR^{*(n)}_{\blambda}(\dots,t^{n-i}u_2,\dots;u_0,u_1;t;p,q)}
= \prod_{\substack{1\le i\le n\\0\le r<s\le 3}} \Gampq(t^{n-i}u_ru_s)
\frac{((p;p)(q;q))^n}{(2\Gampq(t))^n n!} \\ \times \int_{C^n}
\frac{\cR^{*(n)}_{\blambda}(x_1,\dots,x_n;t^{-1/2}u_0,t^{-1/2}u_1;t;p,q)}
     {\cR^{*(n)}_{\blambda}(\dots,t^{n-i-1/2}u_2,\dots;t^{-1/2}u_0,t^{-1/2}u_1;t;p,q)}
\notag\\
 \times \frac{\prod_{1\le i,j\le n} \Gampq(t^{1/2} x_i^{\pm 1}y_j^{\pm 1})}
     {\prod_{1\le i<j\le n} \Gampq(t y_i^{\pm 1}y_j^{\pm 1},x_i^{\pm
         1}x_j^{\pm 1})}\notag\\
\times \prod_{1\le i\le n}
  \frac{\prod_{0\le r<4}\Gampq(t^{-1/2}u_r x_i^{\pm 1})}
       {\Gampq(x_i^{\pm 2})\prod_{0\le r<4}\Gampq(u_r y_i^{\pm 1})}
  \frac{dx_i}{2\pi\i x_i},
\end{multline*}
which is valid under the assumption $t^{n-2}u_0u_1u_2u_3=pq$. Here the notation ``$\dots
t^{n-i}u_2\dots$'' means the set of variables $t^{n-i}u_2$, $i=1,\dots,n$, and similarly for
``$\dots,t^{n-i-1/2}u_2,\dots$''. The contours $C$ are constrained so that every (infinite)
collection of poles which converge to 0 lie inside the contour, and every (infinite) collection of
poles converging to $\infty$ lie outside the contour.

If we reparametrize
\begin{equation*}
(u_0,u_1,u_2,u_3)\mapsto (s,p^{1/2}u_1,tu,p^{1/2}u_3)
\end{equation*}
and take the limit $p\to 0$, the interpolation functions become the (symmetric versions of) the
interpolation polynomials $\bar{P}^{*(n)}_{\lambda}$ of Okounkov \cite{OkBC}, and we obtain the
identity
\begin{multline*}
\frac{\bar{P}^{*(n)}_{\lambda}(y_1,\dots,y_n;q,t,s)}
     {\bar{P}^{*(n)}_{\lambda}(\dots,t^i u,\dots;q,t,s)}
\notag = \frac{\Gamq(t^nus)}{\Gamq(us)} \frac{(q,t;q)^n}{2^n n!}\\ \times \int_{C^n}
\frac{\bar{P}^{*(n)}_{\lambda}(x_1,\dots,x_n;q,t,t^{-1/2}s)}
     {\bar{P}^{*(n)}_{\lambda}(\dots,t^{i-1/2}u,\dots;q,t,t^{-1/2}s)}
\notag\\
\times \frac{\prod_{1\le i,j\le n} \Gamq(t^{1/2} x_i^{\pm 1}y_j^{\pm 1})}
     {\prod_{1\le i<j\le n} \Gamq(t y_i^{\pm 1}y_j^{\pm 1},x_i^{\pm 1}x_j^{\pm 1})}
\prod_{1\le i\le n}
  \frac{\Gamq(t^{-1/2}s x_i^{\pm 1},t^{1/2}u x_i^{\pm 1})}
       {\Gamq(s y_i^{\pm 1},tu y_i^{\pm 1},x_i^{\pm 2})}
  \frac{dx_i}{2\pi\i x_i}.
\end{multline*}
We now want to reparametrize
\begin{equation*}
x_i\to t^{-1/2}sx_i, y_i\to sy_i u\to u/s
\end{equation*}
and take the limit $s\to \infty$ so that the interpolation polynomials become shifted Macdonald
polynomials.  This is an apparently badly behaved limit, as it involves $q$--gamma functions with
arguments tending to infinity.  To fix this, we observe as in Lemma 5.2 of \cite{RainsLimits} that
the $S_n$--invariant function
\begin{equation*}
\frac{\theta_q(\prod_{0\le r\le n+1}w_r/\prod_{1\le i\le n} x_i) \prod_{1\le i\le n} \prod_{0\le
r\le n+1} \theta_q(w_r x_i)} {\prod_{1\le i<j\le n} \theta_q(x_ix_j)\prod_{0\le r<s\le n+1}
  \theta_q(w_rw_s)^{-1}}
\end{equation*}
becomes 1 if we sum over cosets of $S_n$ in the hyperoctahedral  group $BC_n$.  Thus if we multiply
the integrand by $2^n$ times an instance of this function, the integral will be unchanged.  Using
the reflection identity
\begin{equation*}
\Gamq(x)\theta_q(x) = \Gamq(q/x)^{-1}
\end{equation*}
we find that we can cancel the badly scaling gamma factors by taking $w_0 = t^{-1/2}s$, $w_{n+1} =
t^{1/2}u$, and $w_i=t^{1/2}y_i$ for $1\le i\le n$.  In particular, we find
\begin{multline*}
\frac{\bar{P}^{*(n)}_{\lambda}(y_1,\dots,y_n;q,t,s)}
     {\bar{P}^{*(n)}_{\lambda}(\dots,t^i u,\dots;q,t,s)}
= \Gamq(t^nus,q/us) \frac{(q,t;q)^n}{n!} \\ \times \int_{C^n}
\frac{\bar{P}^{*(n)}_{\lambda}(x_1,\dots,x_n;q,t,t^{-1/2}s)}
     {\bar{P}^{*(n)}_{\lambda}(\dots,t^{i-1/2}u,\dots;q,t,t^{-1/2}s)}
\theta_q(us t^{n/2}\prod_{1\le i\le n} y_i/x_i)
\notag\\
\times \prod_{1\le i<j\le n}
  \frac{\Gamq(q/ty_iy_j,q/x_ix_j)}
       {\Gamq(ty_i/y_j,ty_j/y_i,t/y_iy_j,x_i/x_j,x_j/x_i,1/x_ix_j)}
\notag\\
\times \prod_{1\le i,j\le n}
  \frac{\Gamq(t^{1/2}x_i/y_j,t^{1/2}y_j/x_i,t^{1/2}/x_iy_j)}
       {\Gamq(q/t^{1/2}x_iy_j)}\notag\\
 \times \prod_{1\le i\le n}
  \frac{\Gamq(q/sy_i,q/tuy_i,t^{-1/2}s/x_i,t^{1/2}u/x_i,q/x_i^2)}
       {\Gamq(s/y_i,tu/y_i,q/t^{-1/2}s x_i,q/t^{1/2}u x_i,1/x_i^2)}
  \frac{dx_i}{2\pi\i x_i},
\end{multline*}
which after rescaling gives the limit
\begin{multline*}
\prod_{1\le i\le n}
  \frac{(q^{\lambda_i}t^{n+1-i}u;q)}
       {(q^{\lambda_i}t^{n-i}u;q)}
t^{|\lambda|} \bar{P}^{*(n)}_{\lambda}(y_1,\dots,y_n;q,t) = \\ \frac{(q,t;q)^n}{n!} \int_{C^n}
\bar{P}^{*(n)}_{\lambda}(x_1,\dots,x_n;q,t) \frac{\theta_q\left(u t^n\prod_{1\le i\le n}
y_i/x_i\right)}
     {\theta_q(u)}\notag\\
 \times \frac{\prod_{1\le i,j\le n}\Gamq(x_i/y_j,ty_j/x_i)}
     {\prod_{1\le i<j\le n}\Gamq(x_i/x_j,x_j/x_i,ty_i/y_j,ty_j/y_i)}
\notag \prod_{1\le i\le n}
  \frac{\Gamq(q/tuy_i,1/x_i)}{\Gamq(1/y_i,q/ux_i)}
  \frac{dx_i}{2\pi\i x_i},
\end{multline*}
again with the contour containing all ``small'' poles and excluding all ``large'' poles.

Now if we rescale $x\to vx$,$y\to t^{-1}vy$ and take the limit $v\to\infty$, the shifted Macdonald
polynomials become Macdonald polynomials, and we obtain the identity
\begin{align*}
\prod_{1\le i\le n}
  \frac{(q^{\lambda_i}t^{n+1-i}u;q)}
       {(q^{\lambda_i}t^{n-i}u;q)}&
P^{(n)}_{\lambda}(y_1,\dots,y_n;q,t)\notag\\
= \frac{(q,t;q)^n}{n!} \int_{C^n}& P^{(n)}_{\lambda}(x_1,\dots,x_n;q,t) \frac{\theta_q(u
\prod_{1\le i\le n} y_i/x_i)}
     {\theta_q(u)}
\notag\\
& \frac{\prod_{1\le i,j\le n}\Gamq(tx_i/y_j,y_j/x_i)}
     {\prod_{1\le i<j\le n}\Gamq(x_i/x_j,x_j/x_i,y_i/y_j,y_j/y_i)}
\prod_{1\le i\le n}
  \frac{dx_i}{2\pi\i x_i}.
\end{align*}

At this point, we observe that the ``small'' poles are at the points of the form $q^k y_i$, $k\ge
0$, and if we take a residue at one such point, the corresponding poles will not appear in the
residual integrand.  Moreover, if we shrink the contour by ever larger powers of $q$, the integrand
converges to 0 exponentially fast.  We may thus replace the contour--integral by a sum over
residues.  Note that this involves a choice of bijection between the $x$ variables and the $y$
variables, which can be absorbed by symmetry, eliminating the $1/n!$ factor.  We thus obtain the
claimed result of Proposition \ref{NoumiProp}:
\begin{multline*}
\prod_{1\le i\le n}
  \frac{(q^{\lambda_i}t^{n+1-i}u;q)}
       {(q^{\lambda_i}t^{n-i}u;q)}
P^{(n)}_{\lambda}(y_1,\dots,y_n;q,t)\notag = \sum_{\nu\in\N^n} u^{|\nu|} \\ \times \prod_{1\le
i<j\le n}
  \frac{q^{\nu_j}y_j-q^{\nu_i}y_i}{y_j-y_i}
\prod_{1\le i,j\le n}
  \frac{(ty_i/y_j;q)_{\nu_i}}{(qy_i/y_j;q)_{\nu_i}}
P^{(n)}_{\lambda}(q^{\nu_1}y_1,\dots,q^{\nu_n}y_n;q,t).
\end{multline*}
Comparing coefficients of $u^k$ gives the eigenfunction relation of \cite{Noumi}, \cite[Proposition
3.24]{FHHSY}.

\section{Appendix: On a formal version of Theorem \ref{theorem_ascending_multilevel_1}}

\label{section_Different_form_first_operator}

The goal of this section is to obtain a formal version of Theorem
\ref{theorem_ascending_multilevel_1} for the case $r_i=1$ for all $i$.

\begin{theorem}
\label{theorem_observable_multilevel_new_form} Set
$$
 \widehat \O_1(\lambda)=1+ (1-t)\sum_{j=1}^{\infty} (1-q^{\lambda_j}) t^{-j}.
$$
We have:
\begin{multline}
\label{eq_x9}
 \sum_{\lambda^1,\dots,\lambda^N} \MPF_{N,\mathbf{A},\mathbf{B}}(\lambda^1,\dots,\lambda^N) \widehat \O_{1}(\lambda^1)
\cdots
 \widehat \O_{1}(\lambda^N)
 \\= \frac{1}{(2\pi \i)^N} \oint \cdots \oint \prod_{\alpha=1}^N
 \frac{dv_\alpha}{v_\alpha} \prod_{1\le\alpha\le \beta\le N} \HH^{-1}\big((t v_\alpha)^{-1}; B^\beta\big)
   \HH^{-1}\big(A^{\alpha};v_\beta\big)
\W\big((t v_\alpha)^{-1}; v_\beta\big),
\end{multline}
where $v_\alpha$ is integrated over the circle of radius $R_\alpha$ around the origin and
$R_\beta/(tR_\alpha) <1$ for $\alpha<\beta$.
\end{theorem}
{\bf Remark 1.} Formula \eqref{eq_x9} should be understood in the same sense as the statement of
Theorem \ref{theorem_observable_multilevel}.

{\bf Remark 2.} Through suitable specializations and contour deformations Theorem
\ref{theorem_observable_multilevel_new_form} implies the statement of Theorem
\ref{theorem_ascending_multilevel_1} for the case $r_i=1$ for all $i$. In particular, we could also
obtain Corollary \ref{Corollary_q_TASEP} by further setting $t=0$.
\begin{proof}[Sketch of the proof of Theorem \ref{theorem_observable_multilevel_new_form}]

We start from \cite[Proposition 2.2.10]{BigMac} (that is also Theorem
\ref{theorem_ascending_multilevel_1} with $m=1$ and $r=1$)
 which reads for $X=(x_1,\dots,x_N)$ and $Y=(y_1,\dots,y_N)$:
\begin{multline*}
 \sum_{\lambda\in\Y} \left(\sum_{j=1}^N q^{\lambda_j} t^{N-j}\right) P_\lambda(X) Q_\lambda(Y)\\=\Pi(X;Y)\frac{t^N}{(t-1)} \frac{1}{2\pi \i} \oint
 \prod_{j=1}^N \frac{1-t^{-1}z^{-1}x_j}{1-z^{-1}x_j} \prod_{j=1}^N \frac{1-zy_j}{1-t z y_j} \frac{dz}{z}
 \\=\Pi(X;Y)\frac{t^N}{(t-1)} \frac{1}{2\pi \i}\oint
  \HH^{-1}\big((tz)^{-1}; X\big) \HH^{-1}\big(z; Y\big) \frac{dz}{z},
\end{multline*}
where $x_i$ and $y_i$ are assumed to be sufficiently small, and integration goes over the contours
enclosing the poles at $x_i$ and no other poles of the integrand.

Note that the residue of the last integral at $z=0$ is $t^{-N}$. Therefore, using
$$
 \sum_{j=1}^N t^{N-j} =  \frac{1}{1-t}-\frac{t^{N}}{1-t}
$$
we can rewrite (assuming that $q^{\lambda_k}=1$ for $k>N$).
\begin{multline}
\label{eq_new_form}
 \sum_{\lambda\in \Y} \left(\frac{1}{t-1}+\sum_{j=1}^{\infty} (q^{\lambda_j}-1) t^{-j}\right) P_\lambda(X) Q_\lambda(Y)=\frac{\Pi(X;Y)}{2\pi \i(t-1)}\oint
 \HH^{-1}\big((tz)^{-1}; X\big)\HH^{-1}\big(z; Y\big) \frac{dz}{z},
\end{multline}
with integration going over a circle around the origin. Now we can send the number of variables to
infinity in \eqref{eq_new_form} and obtain for two infinite sets of variables $X$, $Y$:
\begin{multline}
\label{eq_new_form_infty}
 \sum_{\lambda\in\Y} \left(1+ (1-t) \sum_{j=1}^{\infty} (1-q^{\lambda_j}) t^{-j}\right) P_\lambda(X) Q_\lambda(Y)\\
 =\Pi(X;Y)\frac{1}{(2\pi \i)}\oint
 \HH^{-1}\big((tz)^{-1}; X\big) \HH^{-1}\big(z; Y\big) \frac{dz}{z}.
\end{multline}
Since \eqref{eq_new_form_infty} has the same form as \eqref{eq_main_lemma} we get an analogue of
Theorem \ref{theorem_observable_multilevel} (which is our Theorem
\ref{theorem_observable_multilevel_new_form}) by repeating same steps as in its proof.
\end{proof}


\begin{thebibliography}{FHHSY}
\bibitem[A]{Amol}
A.~Aggrawal.
\newblock Correlation functions of the Schur process through Macdonald difference operators.
\newblock In preparation.

\bibitem[ACQ]{ACQ}
G.~Amir, I.~Corwin, J.~Quastel.
\newblock Probability distribution of the free energy of the continuum directed random polymer in $1+1$ dimensions.
\newblock {\em Comm. Pure Appl. Math.},{\bf 64}:466--537, 2011.


\bibitem[B]{Borodin} A.~M.~Borodin.
\newblock  Limit Jordan normal form of large triangular matrices over a finite field,
\newblock {\it Functional Analysis and Its Applications,} {\bf 29:4}, 279--281,  1995.


\bibitem[BC]{BigMac}
A.~Borodin, I.~Corwin.
\newblock Macdonald processes
\newblock {\it Probab. Th. Rel. Fields}, to appear. arXiv:1111.4408.

\bibitem[BC2]{BC_discrete}
A.~Borodin, I.~Corwin.
\newblock Discrete time $q$--TASEPs.
\newblock arXiv:1305.2972.

\bibitem[BCF]{BCF}
A.~Borodin, I.~Corwin, P.~Ferrari.
\newblock Free energy fluctuations for directed polymers in random media in $1+1$ dimension.
\newblock {\it Comm. Pure Appl. Math.}, to appear. arXiv:1204.1024.

\bibitem[BCFV]{BCFV}
A.~Borodin, I.~Corwin, P.~L.~Ferrari, B.~Vet\H{o}.
\newblock Stationary solution of 1d KPZ equation. In preparation.

\bibitem[BCR]{BCR}
A.~Borodin, I.~Corwin, D.~Remenik.
\newblock Log-Gamma polymer free energy fluctuations via a Fredholm determinant identity.
\newblock {\it Commun. Math. Phys.}, to appear. arXiv:1206.4573.

\bibitem[BCS]{BCS}
A.~Borodin, I.~Corwin, T.~Sasamoto.
\newblock From duality to determinants for $q$-TASEP and ASEP.
\newblock {\it Ann. Probab.}, to appear. arXiv:1207.5035.

\bibitem[BG]{BG}
A.~Borodin, V.~Gorin.
\newblock Lectures on integrable probability.
\newblock arXiv:1212.3351.


\bibitem[BG2]{BG_GFF}
A.~Borodin, V.~Gorin.
\newblock General $\beta$ Jacobi corners process and the Gaussian free field.
\newblock arXiv:1305.3627.

\bibitem[BO]{BO-z} A.~Borodin, G.~Olshanski.
\newblock Z-measures on partitions and their scaling limits.
\newblock {\it  European J. Combin. }, {\bf 26,} no.\ 6: 795--834, 2005.

\bibitem[BP]{BP}
A.~Borodin, L.~Petrov.
\newblock Nearest neighbor Markov dynamics on Macdonald processes.
\newblock arXiv:1305.5501.

\bibitem[COSZ]{COSZ}
I.~Corwin, N.~O'Connell, T.~Sepp\"{a}l\"{a}inen, N.~Zygouras.
\newblock Tropical combinatorics and Whittaker functions.
\newblock arXiv:1110.3489.

\bibitem[CP]{CP}
I.~Corwin, L.~Petrov.
\newblock Two-sided $q$-PushASEP. In preparation.

\bibitem[GLO]{GLO}
A.~Gerasimov, D.~Lebedev, S.~Oblezin.
\newblock On $q$-deformed $\mathfrak{gl}_{\ell+1}$--Whittaker
functions I,II,III.
\newblock  Comm. Math. Phys. 294 (2010), 97--119, arxiv:math.RT/0803.0145;
Comm. Math. Phys. 294 (2010), 121--143, arxiv: math.RT/0803.0970; arxiv: math.RT/0805.3754.

\bibitem[FHHSY]{FHHSY}
B. Feigin, K. Hashizume, A. Hoshino, J. Shiraishi, S. Yanagida.
\newblock A commutative algebra on degenerate $CP^1$ and Macdonald polynomials.
\newblock {\it J. Math. Phys.} {\bf 50}, 095215, 2009.

\bibitem[FR]{FR}
P.~ J.~Forrester, E.~Rains.
\newblock Interpretations of some parameter dependent generalizations
of classical matrix ensembles.
\newblock {\it Probab. Th. Rel. Fields,} {\bf 131}:1--61, 2005.

\bibitem[F1]{F1}
J.~Fulman.
\newblock Probabilistic measures and algorithms arising from the Macdonald symmetric
functions.
\newblock arXiv:math/971223.

\bibitem[F2]{Fulman}
J.~Fulman.
\newblock Random matrix theory over finite fields.
\newblock {\it Bull. Amer. Math. Soc.} {\bf 39}:51--85, 2001

\bibitem[GKV]{GKV} V.~Gorin, S.~Kerov, A.~Vershik.
\newblock Finite traces and representations of the group of infinite matrices over a finite
field.
\newblock arXiv:1209.4945.


\bibitem[GS]{GS}
V.~Gorin, M.~Shkolnikov.
\newblock Multilevel Dyson Brownian Motions via Jack polynomials, in preparation.

\bibitem[Ke1]{Kerov} S.~Kerov.
\newblock The Boundary of Young Lattice and Random Young Tableaux.
\newblock In: {\it DIMACS Series in Discrete Mathematics and Theoretical Computer Science,} {\bf 24},
133--158, 1996.

\bibitem[Ke2]{Kerov_book} S.~Kerov.
\newblock {\it Asymptotic Representation Theory of the Symmetric Group and its Applications in
Analysis.}
\newblock Amer. Math. Soc., Providence, RI,  2003.


\bibitem[KOO]{KOO} S.~Kerov, A.~Okounkov, G.~Olshanski.
\newblock  The boundary of Young graph with Jack
edge multiplicities.
\newblock {\it Internt. Math. Res. Notices}  {\bf 4}:173--199, 1998.

\bibitem[Ki]{Kingman} J.~F.~C.~Kingman.
\newblock  Random partitions in population genetics.
\newblock {\it Proc. R. Soc. London, A,}
{\bf 361}:1--20, 1978.


\bibitem[K]{Kr} C.~Krattenthaler,
\newblock Advanced determinant calculus.
\newblock S\'eminaire Lotharingien Combin. 42 (``The Andrews Festschrift'') (1999), Article B42q, 67 pp.

\bibitem[M]{M}
I.~G.~Macdonald.
\newblock {\it Symmetric Functions and Hall Polynomials.}
\newblock 2nd ed. Oxford University Press, New York. 1999.

\bibitem[NS]{Noumi} M.~Noumi, A.~Sano.
\newblock An infinite family of higher-order difference operators that commute with Ruijsenaars operators of type A.
\newblock In Preparation.

\bibitem[OC]{OCon}
N.~O'Connell.
\newblock Directed polymers and the quantum Toda lattice.
\newblock  {\em Ann. Probab.}, {\bf 40}:437--458, 2012.

\bibitem[OCY]{OConYor}
N.~O'Connell, M.~Yor.
\newblock Brownian analogues of Burke's theorem.
\newblock {\it Stoch. Proc. Appl.}, {\bf 96}:285--304, 2001.

\bibitem[O1]{Ok}
A.~Okounkov.
\newblock Infinite wedge and random partitions.
\newblock {\it Selecta Math.~(N.~S.)} {\bf 7}:51--81, 2001.

\bibitem[O2]{OkBC}
A.~Okounkov
\newblock BC-type interpolation Macdonald polynomials and binomial formula for Koornwinder polynomials.
\newblock {\it Trans. Groups}, {\bf 3}:181--207, 1998.

\bibitem[OO]{OkOl} A.~Okounkov, G.~Olshansky.
\newblock Asymptotics of Jack polynomials as the number of variables goes to infinity.
\newblock {\it International Mathematics Research Notices}, {\bf 13}:
641--682, 1998.

\bibitem[OR]{OR_Schur}
A.~Okounkov, N.~Reshetikhin.
\newblock Correlation function of Schur process with application to local geometry of a random 3-dimensional Young diagram.
\newblock {\it J. Amer. Math. Soc.}, {\bf 16}:581--603, 2003.

\bibitem[P]{Petrov}
L.~Petrov.
\newblock A two-parameter family of infinite-dimensional diffusions in the Kingman simplex.
\newblock {\it Funct. Analys. Appl.}, {\bf 43}:279--296, 2009.

\bibitem[R1]{RainsTrans}
E.~Rains.
\newblock Transformations of elliptic hypergeometric integrals.
\newblock {\it Ann. Math.}, {\bf 171}:169--243, 2010.

\bibitem[R2]{RainsLimits}
E.~Rains.
\newblock Limits of elliptic hypergeometric integrals.
\newblock {\it Ramanujan J.}, {\bf 18}:257--306, 2009.

\bibitem[SS]{SS}
T.~Sasamoto, H.~Spohn.
\newblock One-dimensional KPZ equation: an exact solution and its universality.
\newblock {\em Phys. Rev. Lett.}, {\bf 104}:23, 2010.


\bibitem[Se]{S}
T.~Sepp\"{a}l\"{a}inen.
\newblock Scaling for a one-dimensional directed polymer with boundary conditions.
\newblock {\it Ann. Probab.}, {\bf 40}:19--73, 2012.

\bibitem[Sh]{Shi} J.~Shiraishi.
\newblock A Family of Integral Transformations and Basic Hypergeometric Series.
\newblock {\it Comm. Math. Phys}, {\bf 263}:439--460, 2006.

\bibitem[V]{Vul} M.~Vuleti\'c.
\newblock A generalization of MacMahon's formula.
\newblock {\it Trans. Amer. Math. Soc.}, {\bf 361}:2789-2804, 2009.

\bibitem[W]{Warnaar}
S. Ole Warnaar.
\newblock Bisymmetric functions, Macdonald polynomials and $sl_3$ basic hypergeometric series.
\newblock {\it Compos. Math.}, {\bf 144}:271--303, 2008.

\bibitem[Z]{Zel}
A.~Zelevinsky.
\newblock {\it Representations of finite classical groups. A Hopf algebra approach}.
\newblock Lecture notes in mathematics, 869. Springer-Verlag, Berlin-New York, 1981, 184 pp.

\end{thebibliography}
\end{document}